\documentclass[reqno]{amsart}
\usepackage{graphicx}
\usepackage{amssymb,amsmath,amsfonts ,amsthm,enumitem}
\usepackage{ifthen,mathrsfs,verbatim,epsfig,color,pstricks}
\usepackage[mathcal]{euscript}

\usepackage[numbers,sort&compress]{natbib}
\usepackage{stmaryrd}    

\usepackage{caption}

\usepackage[pdfencoding=auto, psdextra]{hyperref}
\hypersetup{colorlinks=true,linkcolor=green,pdfborderstyle={/S/U/W 1}}

 \usepackage{background}
 \SetBgContents{Preprint version of submitted article.}
 \SetBgScale{1}
 \SetBgAngle{0}
 \SetBgOpacity{1}
 \SetBgPosition{current page.south}
 \SetBgVshift{1cm}

\theoremstyle{plain}
\newtheorem{theorem}  {Theorem}  [section]
\newtheorem{lemma}  [theorem]   {Lemma}
\newtheorem{corollary}[theorem] {Corollary}
\newtheorem{construction} [theorem] {Construction}
\newtheorem{fact} [theorem] {Fact}
\newtheorem{proposition} [theorem] {Proposition}

\newtheorem*{claim*} {Claim}

\newtheorem*{problem*} {Problem}

\theoremstyle{definition}

\newtheorem{definition}[theorem] {Definition}
\newtheorem{example}[theorem] {Example}
\newcommand{\instance}[1]{ ~\ \par
 Instance: #1 \\ }
\newcommand{\decision}[1]{ ~\ \indent
 Decision: #1 }

\newcommand{\claimproof}{\renewcommand{\qedsymbol}{$\diamond$}}
\newcommand{\exqed}{\hfill $\triangle$}

\makeatletter
\newcommand{\mylabel}[2]{#2\def\@currentlabel{#2}\label{#1}}
\makeatother

\usepackage{tikz}
\usetikzlibrary{hobby,patterns,calc,positioning}   
\tikzset{
  vert/.style={circle, draw=black!100,fill=black!100,thick, inner sep=0pt, minimum size=2mm, anchor=center}, 
  empty/.style={circle,  fill=none},
  labvert/.style={circle, draw=black!100,fill=none,thick, inner sep=2pt, minimum size=2mm},
 cycle/.style={circle, draw=black!100, fill=none, dotted, inner sep=2pt, minimum size=7mm},
 clump/.style={circle, draw=gray!80, fill=none, minimum size = 2cm},
 square/.style={rectangle, draw=black, fill=none, minimum size = 1cm, rounded corners}   
}

\pgfdeclarelayer{bg2}
\pgfdeclarelayer{bg1}
\pgfdeclarelayer{fg1}
\pgfdeclarelayer{fg2}
\pgfsetlayers{bg2,bg1,main,fg1,fg2}

\newcommand{\fsquare}[5][gray!80]{\filldraw[draw=black, fill=#1] ({#2}.center) -- ({#3}.center) -- ({#4}.center) -- ({#5}.center) -- cycle; \draw (#2) -- (#4); \draw (#3) -- (#5); }   
\newcommand{\vsquare}[5][gray!80]{\fsquare[#1]{#2#4}{#3#4}{#3#5}{#2#5}}
\newcommand{\hsquare}[5][gray]{\filldraw[draw=black, fill=#1] ({#2}.center) -- ({#3}.center) -- ({#4}.center) -- ({#5}.center) -- cycle; \draw (#2) -- (#4); }    
\newcommand{\vhsquare}[5][gray]{\hsquare[#1]{#2#4}{#3#4}{#3#5}{#2#5}}

\usepackage{pgffor}
\foreach \x in {A,...,Z}{%
\expandafter\xdef\csname bb\x\endcsname{\noexpand\ensuremath{\noexpand\mathbf{\x}}}
\expandafter\xdef\csname mbb\x\endcsname{\noexpand\ensuremath{\noexpand\mathbb{\x}}}
\expandafter\xdef\csname cc\x\endcsname{\noexpand\ensuremath{\noexpand\mathcal{\x}}}
\expandafter\xdef\csname ss\x\endcsname{\noexpand\ensuremath{\noexpand\mathscr{\x}}}
}
\foreach \x in {a,...,z}{%
\expandafter\xdef\csname bb\x\endcsname{\noexpand\ensuremath{\noexpand\mathbf{\x}}}
\expandafter\xdef\csname cc\x\endcsname{\noexpand\ensuremath{\noexpand\mathcal{\x}}}
\expandafter\xdef\csname ss\x\endcsname{\noexpand\ensuremath{\noexpand\mathscr{\x}}}
}
\foreach \x in {C,F,N,R,Z}{%
\expandafter\xdef\csname \x\endcsname{\noexpand\ensuremath{\noexpand\mathbb{\x}}}}

 \newcommand{\angles}[1]{\langle #1 \rangle}
\newcommand{\goes}[1]{\buildrel{#1}\over\longrightarrow}
\newcommand{\iso}[1][]{\buildrel {#1} \over \cong}

\newcommand{\PWF}[4]{
  \left\{ \begin{array}{ll}
    #1 & \mbox{ if } #2 \\
    #3 & \mbox{ if } #4 \\
  \end{array}\right.}

\DeclareMathOperator{\bHom}{\mathbf{Hom}}
  
\DeclareMathOperator{\NP}{NP}
\DeclareMathOperator{\Recol}{Recol}
\DeclareMathOperator{\Recon}{Recon}
 \DeclareMathOperator{\NF}{NonFlat} 
 \DeclareMathOperator{\img}{Im} 
\DeclareMathOperator{\Mix}{Mix}
\DeclareMathOperator{\Hom}{Hom}
\DeclareMathOperator{\coNP}{coNP}
\DeclareMathOperator{\Poly}{P}
\DeclareMathOperator{\PSPACE}{PSPACE}

 \newcommand{\A}{Z}
  
 \newcommand{\DA}{\bbA}
 \newcommand{\DB}{\bbB}
 \newcommand{\DC}{\bbC}
 
 \newcommand{\DE}{\bbE}
 \newcommand{\DF}{\bbF}
 \newcommand{\DX}{\bbX}
 \newcommand{\DsC}{\bbC}
 \newcommand{\DsH}{\bbH}
 \renewcommand{\DH}{\bbH}
 \newcommand{\DG}{\bbG}

 \newcommand{\Gb}{G_*}
 \newcommand{\sH}{H}
 \newcommand{\ZH}{Z_H}
 \newcommand{\sB}{sZ}
 \newcommand{\sT}{T_*}
 \newcommand{\sZ}{Z_*}

 \newcommand{\vtimes}{\times}
 \newcommand{\ZZ}{{\mbbZ}}
 \newcommand{\ZS}[1][\Sigma]{Z_{#1}}

 \renewcommand{\tilde}{\widetilde}

\makeatletter
\newcommand{\Ga}[1][\@nil]{%
  \def\tmp{#1}%
   \ifx\tmp\@nnil
       G'_{*}
    \else
       G_{*}^{(#1)}
    \fi}
\makeatother

\begin{document}

\author{Hyobeen Kim}
\address{Kyungpook National University, Republic of Korea}
\email{hbkim@knu.ac.kr}

\author{Jae-baek Lee}
\address{University of Victoria, Canada}
\email{dlwoqor0923@uvic.ca}

\author{Mark Siggers}
\address{Kyungpook National University, Republic of Korea}
\email{mhsiggers@knu.ac.kr}
\thanks{The last author was supported by the Korean NRF Basic Science Research Program (NRF-2022R1A2C1091566) funded by the Korean government (MEST),
and by the Kyungpook National University Research Fund.}

\title[Mixing for reflexive graphs]{Mixing is hard for triangle-free reflexive graphs}


  %

\begin{abstract}
  In the problem ${\rm Mix}(H)$ one is given a graph $G$ and must decide if the Hom-graph
  ${\rm {\bf Hom}}(G,H)$ is connected.  We show that if $H$ is a triangle-free reflexive graph with at least one cycle, ${\rm Mix}(H)$ is ${\rm coNP}$-complete.  The main part of this is a reduction to the problem ${\rm NonFlat}({\rm{\bf H}})$ for a simplicial complex ${\rm{\bf H}}$, in which one is given a simplicial complex ${\rm{\bf G}}$ and must decide if there are any simplicial maps $\phi$ from ${\rm{\bf G}}$ to ${\rm{\bf H}}$ under which some $1$-cycles of ${\rm{\bf G}}$ maps to homologically nontrivial cycle of ${\rm{\bf H}}$.  We show that for any reflexive graph $H$, if the clique complex ${\rm{\bf H}}$ of $H$ has a free, nontrivial homology group $H_1({\rm{\bf H}})$, then ${\rm NonFlat}({\rm{\bf H}})$ is ${\rm NP}$-complete.     
\end{abstract}

\maketitle

\section{Introduction}

In the last 15 years there have been several papers looking at the `reconfiguration' and `mixing' variations of the classical graph colouring, and of the more general graph homomorphism, or $H$-colouring, problem.  See \cite{Nishimura} for a survey of recent results in the area.

 These variations of the $H$-colouring problem have a common presentation in terms of the `Hom-graph'. 
A {\em homomorphism} $\phi:G \to H$ from a graph $G$ to a graph $H$ is an edge preserving vertex map. As a homomorphism from a graph $G$ to $K_k$ is a $k$-colouring of $G$, a homomorphism to $H$ is also called an $H$-colouring of $G$.  Given graphs $G$ and $H$, $\Hom(G,H)$ is the set of $H$-colourings of $G$. The $\Hom$-graph $\bHom(G,H)$ is the graph on the vertex set $\Hom(G,H)$
in which two vertices $\phi, \psi \in \Hom(G,H)$ are adjacent if for all edges $uv$ of $G$,
$\phi(u)\psi(v)$ is an edge of $H$.

While the well known $H$-colouring problem $\Hom(H)$ asks for an instance graph $G$ if the Hom-graph $\bHom(G,H)$ contains any vertices, the $H$-mixing problem asks if $\Hom(H)$ is connected.

\begin{problem*}[$\Mix(H)$]\label{prob2} 
 \instance{A graph $G$.} 
 \decision{Is $\bHom(G,H)$ connected?} 
\end{problem*}

The $H$-reconfiguration (or $H$-recolouring\footnote{The two problems, defined differently, coincide for graph, but not for digraphs. The distinction is addressed in \cite{BLS21}).}) problem asks if two vertices of the graph $\bHom(G,H)$ are in the same component. 

\begin{problem*}[$\Recon(H)$]\label{prob1}
\instance{A graph $G$ and $f,g\in \Hom(G,H)$.}
\decision{Is there a path in $\bHom(G,H)$ between $f$ and $g$?}
\end{problem*}

Homomorphisms, the Hom-graph, and the above problems can be defined for all relational structures, not just graphs, and it is well known (\cite{BulCSP},\cite{Zhuk}) that the problem $\Hom(H)$ is in $\NP$ for any finite structure $H$, and is either in $\Poly$ or is $\NP$-complete. The problems $\Recon(H)$ and $\Mix(H)$ are not generally in $\NP$, but in $\PSPACE$.  Indeed a path between homomorphisms can be verified one edge at a time in polynomial space, showing that $\Recon(H)$ is in $\PSPACE$ and observing that this can be done for every pair of homomorphisms shows that $\Mix(H)$ is also in $\PSPACE$.

In \cite{GKMP09}, these problems were addressed for binary relational structures $H$. In this setting the $H$-reconfiguration problem was shown to exhibit dichotomy.  Depending on $H$, $\Recon(H)$ is either polynomial time solvable or $\PSPACE$-complete. Once $\Recon(H)$ is known to be polynomial time solvable for a given structure $H$, the problem $\Mix(H)$ drops into the complexity class $\coNP$, as we now have a polynomial certificate of the disconnectedness of $\bHom(G,H)$.  In \cite{GKMP09}, the authors went on to conjecture a trichotomy for $\Mix(H)$ for binary relational structures $H$: depending on $H$, $\Mix(H)$ is polynomial time solvable, $\coNP$-complete, or $\PSPACE$-complete. They conjecture a classification for the trichotomy, and gave strong evidence supporting it.

The expectation is that with graphs $H$ we should have a similar complexity breakdown--
a dichotomy between $\Poly$ and $\PSPACE$-complete for $\Recon(H)$, and a trichotomy between $\Poly$, $\coNP$-complete, and $\PSPACE$-complete for $\Mix(H)$.  The graphs for which the problems are $\PSPACE$-complete should coincide. Though it seems premature to suggest a classification of the graphs $H$ for which these problems fall into the respective complexity classes, some discussion of this is given in \cite{LNS20} for $\Recon(H)$. The proofs in \cite{LNS20} are very topological, and suggest the topology of $H$ may play an important role in the complexity of $\Recon(H)$.  Indeed, the Hom-graph was originated to apply topological ideas to graph colouring problems, so topological methods are perhaps inevitable.

There are few results about the complexity of mixing for graphs. Recall that a graph in which every vertex has a loop is {\em reflexive} and a graph in which no vertices have loops is {\em irreflexive}. For irreflexive graphs, it was shown in \cite{CvdHJ09} that $\Mix(K_3)$ is $\coNP$-complete, and recently in \cite{BM23} this was extended to all odd cycles.  Nothing yet has been proved for reflexive graphs, but there is one basic result from the literature. From \cite{BW00}, we have that if a reflexive graph $H'$ dismantles to $H$, then $\bHom(G,H')$ has the same number of components as $\bHom(G,H)$ for all $G$. So the complexity of $\Mix(H')$ is the same as that of $\Mix(H)$.  From this we conclude  that $\Mix(H)$ is trivial for any (reflexive) dismantlable graph $H$. 

 In the case of the reconfiguration problem one finds a parallel between results for reflexive graphs and irreflexive graphs, and the topological techniques used are `cleaner' in the reflexive setting, so we start there.  The result from \cite{BM23} that $\Mix(C)$ is $\coNP$-complete for any irreflexive odd cycle $C$ suggests, through this parallel, that $\Mix(C)$ should also be $\coNP$-complete for $C$ being any reflexive cycle of girth at least $4$.  Our goal was to prove this, but in doing so we found the proof could be extended, with not too much work to all triangle free reflexive graphs.  

In \cite{LNS21}, results of \cite{Wroch15} were adapted to reflexive graphs, and it was shown that $\Recon(H)$ is polynomial time solvable (for reflexive instances) for any triangle-free reflexive graph $H$.  As discussed above, this yields the following corollary. 

   \begin{fact}\label{fact:coNP}
    For any triangle-free reflexive graph $H$, the problem $\Mix(H)$ is in the complexity class $\coNP$.
   \end{fact}
   
  Clearly if $H$ is disconnected then $\Hom(G,H)$ is disconnected as soon as $G$ contains an edge, so we always assume that $H$ is connected.  $\Mix(H)$ can then be solved for an instance $G$ component-wise, so we will also assume that $G$ is connected.  If $H$ is a reflexive tree, then it is dismantlable,  in which case $\Mix(H)$ is trivial. In this paper, we prove completeness when $H$ has a nontrivial cycle, with the following theorem.  
 
  \begin{theorem}\label{thm:full}
    For any triangle-free reflexive non-tree $H$, the problem $\Mix(H)$ is $\coNP$-complete. 
  \end{theorem}

The proof depends on results from \cite{LNS21} and exploits the same connection to topology that is central to the algorithms there and in \cite{Wroch15}. This lends support to the notion that topology may factor into the trichotomy classification.  We expect the proofs and results of this paper should yield somewhat messier irreflexive versions depending on \cite{Wroch15} rather than on \cite{LNS21}, and digraph versions depending on \cite{LMS23}, which extends \cite{Wroch15} to digraphs.  
  
 Finally, we note here, that our main technical theorem, Theorem \ref{thm:NF}, implies that $\Mix(H)$ is $\coNP$-Hard for a much larger class of reflexive graphs than triangle-free non-trees,  but we do not have a corresponding proof that $\Recol(H)$ is in $\Poly$ unless $H$ is triangle-free, so we do not know that $\Mix(H)$ is in $\coNP$ for these graphs.   
 

 \section{Outline of Proof and of Main Construction}\label{sect:outline} 
 
  One of the essential features in many papers on recolouring and homomorphism reconfiguration is the simple fact that when mapping a cycle $B$ to another smaller cycle $C$ of girth at least $4$, (or $5$ if $B$ is not reflexive) the number of times that $B$ winds around $C$ is invariant under reconfiguration.   In \cite{Wroch15}, Wrochna, to great effect, stated this in terms of homotopy, and showed that the homotopy, appropriately defined, of cycles in $G$ when mapped to $H$, is invariant under reconfiguration.
 This idea was adapted to reflexive graphs in \cite{LNS21}, and digraphs in \cite{LMS23}. 
  In this paper we use the sightly coarser, but computationally simpler topological invariant of homology.  We define homology properly in Section \ref{sect:homology}, but for the purposes of this outline, one unfamiliar with the concept benefits by thinking of $H$ as a cycle, and the `homology' of another  cycle $C$ mapping to $H$ as being the number of times it winds around $H$; having trivial homology means it does not wind.   
  
  As $H$ is reflexive, there are always constant maps in $\Hom(G,H)$, and as $H$ is connected, it is easy to see that these are always in the same component of $\Hom(G,H)$. So to decide if $\Hom(G,H)$ is connected, one must decide if every $H$-colouring of $G$ reconfigures to a constant map; this follows the basic approach used in \cite{BM23}.   In \cite{LNS21} it was shown that an $H$-colouring of $G$ reconfigures to a constant map if and only if it has trivial homotopy (see Corollary \ref{cor:LNS} below), and trivial homotopy is implied by trivial homology (see Lemma \ref{lem:homhom}). So deciding if all $H$-colourings of $G$ have trivial homology is a sub-problem of $\Mix(H)$.  Calling an $H$-colouring of $G$ {\em flat} if it has trivial homology, we reduce our task  to showing $NP$-completeness of the problem $\NF(H)$ to deciding if there is a non-flat $H$-colouring of an instance $G$ (see Lemma \ref{lem:3toNF}).  
  
  To prove this, we will reduce $3$-colouring to $\NF(H)$ via a gadget construction. Given an instance $G$ of $3$-colouring, we construct an instance $\Gb$ of $\NF(H)$ such that $G$ has a $3$-colouring if and only if $\Gb$ has a non-flat homomorphism to $H$. Now, we give an overview of the construction.    
   
   The (categorical) product $A \times B$ of two graphs $A$ and $B$ is the graph with vertex set $V(A) \times V(B)$, in which $(a,b) \sim (a',b')$ if $a \sim a'$ in $A$ and $b \sim b'$ in $B$.   
   For a vertex $a$ of $A$, $a \times B$ is the subgraph induced on the vertices $\{ (a,b) \mid b \in B \}$. If $A$ has a loop, then this is an isomorphic copy of $B$.
    
   Where $P_\ell$ is the reflexive path on vertices ${0, \dots, \ell}$, it is well known, and easy to show, that the existence of a path $P_\ell$ from $\phi$ to $\psi$ in $\Hom(G,C)$ is equivalent to the existence of a homomorphism from $P_\ell \times G$ to $C$ that restricts on $0 \times G$ to $\phi$ and on $\ell \times G$ to $\psi$. For a cycle $C$, we refer to $P_\ell \times C$ as a {\em path-of-cycles}  and to each copy $i \times C$ of $C$ a {\em slice} of the path. 
   
   In Fact \ref{fact:preshom}, we observe the basic fact that under a homomorphism of a path-of-cycles to $H$, all slices have the same homology.  This property of paths-of-cycles is the main building block in our construction of $\Gb$.   Let $Z$ be a smallest cycle in $H$.  The graph $\Gb$ will contain a special copy $\sZ$ of $Z$, and for every vertex $v$ of $G$, three more copies $\A_0^v, \A_1^v$ and $\A_2^v$ of $Z$. 
   
   Using a gadget $S_3$ made of paths-of-cycles we will ensure that under a homomorphism $\phi:\Gb \to H$, if $\sZ$ has nontrivial homology, then exactly one of the $\A_i^v$ for each $v$ have the same homology as $\sZ$, and the other two have trivial homology.  For every edge $uv$ of $G$ we use another path-of-cycles gadget $S_2$ to ensure that not both  $\A_i^v$ and $\A_i^u$ get nontrivial homology. From this it will not be hard to show that if there is a non-flat homomorphism $\phi: \Gb \to H$; one such that $\sZ$ has nontrivial homology, then we get a $3$-colouring of $G$ by mapping $v$ to the $i$ such that $\A_i^v$ has nontrivial homology.    
   
    
   The construction as described above is not so difficult. The gadgets are constructed in Section \ref{sect:windsum} and the rest of the construction is done with Construction \ref{const:G*} of Section \ref{sect:stage1}. However, we also have to show that if there is no $3$-colouring of $G$, then all homomorphism of $\Gb$ to $H$ are flat.  It is not difficult to show that if there is no $3$-colouring of $G$, then any homomorphism of $\Gb$ to $H$ has trivial homology on the slices of our path-of-cycles gadgets. This does not yet ensure that all $H$-colourings of $\Gb$ are flat. There may be a lot of other `non-gadget' cycles in $\Gb$ that also must have trivial homology.  
   
   To deal with this we define a subgraph $T_*$ of $\Gb$ that contains a basis, in the cycle space, of all of these non gadget-cycles.  For each cycle in $T_*$, we  `plug' it, by triangulating it to a new central vertex.  So that this does not break the necessary properties of $\Gb$, we must build $\Gb$ and the gadgets so that we know the possible colourings on the vertices of $T_*$.  Because of this, our gadget construction mentions a subgraph $T$ the copies of which will become part of $T_*$ in $\Gb$.  The proof that the cycles in $T_*$ that we plug in this way are all the cycles that have to be plugged takes considerable work, using a lot of topological machinery.  

We try, as much as possible, to remove these topological proofs from the main flow of the proof.  They take considerable patience for a reader unfamiliar with them. However, we also found that they provide a nice intuitive motivation for such basic results of a first algebraic topology class as the K\"{u}nneth formula, excision, and Mayer-Vietoris sequences.  
   
In Section \ref{sect:homology} we recall the basic definitions of homology that are necessary throughout the paper. In  Section \ref{sect:non-flat} we reduce the mixing problem to the non-flat problem and state our main theorem about the non-flat problem.  In Sections \ref{sect:windsum} and \ref{sect:stage1} we give our construction and prove the main theorem modulo a result about the basis of its cycle space that we prove in Section \ref{sect:NTbasis}.

  \section{The homology of a graph}\label{sect:homology} 
  
   We recall the basic definitions of simplicial homology with coefficients in $\ZZ$. This is basically Section 2.1 of Hatcher \cite{Hatcher} applied to the clique-complex of a graph $G$. 
   
   In the {\em clique-complex $\DG$} of a graph $G$, an {\em $n$-simplex} is an ordered $(n+1)$-clique of $G$. As $G$ is reflexive, we consider multisets $[v,v,v]$, and $[u,v,u]$ when $u \sim v$, to be $3$-cliques. 

 The {\em chain group} $C_n(\DG)$ of $G$ is the group generated as a $\ZZ$-module over the set of $n$-simplices of $\DG$; $C_{-1}(\DG)$ is taken to be $0$. The elements of the groups $C_n(\DG)$ are $n$-chains-- they can be view as multisets of the of $n$-simplices of $\DG$.
   Between these groups we have the boundary maps $\delta_n :C_n(\DG) \to C_{n-1}(\DG)$ defined on $n$-simplices by
   \[ \delta_n([v_0,v_1, \dots, v_n]) = \sum_{i=1}^n (-1)^{i}[v_0,v_1, \dots, v_{i-1}, v_{i+1}, \dots, v_n], \]
   and extended to the whole group additively.  For example, assume that $G$ has two $3$-cliques, $\{a,b,c\}$ and $\{a,d,b\}$.  The sum, $[a,b,c] + [a,d,b]$ say, is a $2$-chain.  Its boundary is 
      \begin{eqnarray*}
       \delta( [a,b,c] + [a,d,b] )  & = &  [b,c] - [a,c] + [a,b]  + [d,b] - [a,b] + [a,d] \\
                                    & = &  [a,d] + [d,b] + [b,c] - [a,c].\\
     \end{eqnarray*}

    Homology groups are defined using the boundary maps as follows.   
    The co-kernel  
     $B_{i}(\DG):= \delta_{i+1}(C_{i+1}(\DG)) \leq C_i(\DG)$ 
    of $\delta_{i+1}$ is the {\em $i^{th}$ boundary group}; its elements are $i$-boundaries.
    The kernel $Z_i(\DG) = \ker(\delta_{i}) \leq C_{i}(\DG)$ is the
    {\em $i^{th}$ cycle group}; its elements are {\em $i$-cycles}.
    It is easy to check that $\delta_i \circ \delta_{i+1} = 0$ and so $B_i(\DG) \leq Z_i(\DG)$.
    The {\em $i^{th}$ homology group} is $H_i(\DG) = Z_i(\DG)/B_i(\DG)$; its elements are classes of {\em homologically equivalent} cycles, though we often refer to these classes simply as cycles. 
    
    We start with some standard easy calculations.
    \begin{example}
      For any connected graph $G$, $H_0(\DG) = \ZZ$.  
      Indeed, $C_0(\DG) = \angles{ [v] \mid v \in V(G) }$ and as $\delta_0$ is the $0$ map, 
      $Z_0(\DG) = \ker{\delta_0} = C_0(\DG)$.  On the other hand, $C_1(\DG) = \angles{ [u,v] \mid uv \in E(G) }$, so $B_0(\DG) = \delta_1(C_1(\DG)) = \angles{ [u] - [v] \mid uv \in E(G) }$.
      Thus $H_0(\DG) = Z_0(\DG)/B_0(\DG) = \ZZ$.
    \exqed
    \end{example}

    \begin{example}\label{example:basprop}
     The following are all true for any triangle-free reflexive $G$. 
    \begin{enumerate}
        \item $C_2(\DG)$ is the free group generated by the set 
      \[ \{ [v,v,v] \mid v \in V(G)\} \cup \{ [u,u,v], [u,v,u], [v,u,u] \mid uv \in E(G)\} \]
        \item As $\delta([v,v,v]) = [v,v]$ and  $\delta([u,v,u]) = [v,u] - [u,u] + [u,v]$,  one finds that $B_1(\DG) = \delta_2(C_2)$ is free generated by  \[ \{ [v,v] \mid v \in V(G) \} \cup \{ [u,v] + [v,u] \mid uv \in E(G) \}.  \]
        It follows that $[v,u] = -[u,v]$ in $H_1(\DG)$ for every edge $uv$ of $G$. 
        \item $Z_1(\DG)$ is free generated by cycles (including loops and $2$-cycles) in $G$.
        \item $H_1(\DG) = Z_1(\DG)/B_1(\DG)$ is free generated by the undirected edge-sets of cycles (not-including loops and $2$-cycles) in $G$.
        \item $H_n(\DG)$ is {\em trivial}; that is, has only the $0$ element; for $n \geq 2$.  
    \end{enumerate}
    \exqed
    \end{example}
    
    In light of (3), we will often talk of a cycle $(x_0, x_1,  \dots, x_\ell, x_0)$ of $G$ as being in $Z_1(\DG)$. This really means the $1$-cycle
      \[ [x_0,x_1] + [x_1,x_2] + \dots + [x_\ell, x_0]. \] 
    
    \begin{example}
    
    If $G$ is not triangle free, then a triangle such as $(u, v, w, u)$, which lives
    in $Z_1(\DG)$ as $\sigma = [u,v] + [v,w] + [w,u]$ is a boundary in $B_1(\DG)$: 
         \[ \delta_2([u,v,w]) = [v,w] - [u,w] + [u,v] = [u,v] + [v,w] + [w,u], \]
    so $[\sigma] = 0$ in $H_1(\DG)$.  Thus we still have that $H_1(\DG)$ is generated by undirected edges-sets of cycles in $G$, but it may no longer be free, and there can be more complicated boundaries as well. The groups $H_n$ for $n \geq 2$ are also not generally trivial, if $G$ is not triangle-free.
    
    \exqed
    \end{example}
    
    The group $Z_1(\DG)$ of $1$-cycles is what is often called the integral cycle space of $G$; as opposed to the cycle space, which we get with the same development using $\ZZ_2$ in place of $\ZZ$. 
    A cycle $C$ in $G$ is {\em trivial} if $[C] = 0$ in $H_1(\DG)$, where $0$ represents the class containing the empty cycle. If $H_1(\DG) = \{ 0 \}$ then we say $G$ has {\em trivial homology}. 

    
    The following is the basic statement that all slices of a path of cycle get the same homology under any homomorphism. 
    \begin{fact}\label{fact:preshom} 
       Let $\phi: P_\ell \times C \to H$ be a homomorphism.  For all $j,k \in \{0,1, \dots, \ell\}$ we  have $[\phi(j \times C)] = [\phi(k \times C)]$ in $H_1(\DH)$.
    \end{fact}
    \begin{proof}
      Clearly it is enough to prove this for $j=0$ and $k=1$.
      Observe that the boundary $\delta_2(c_i)$ of the $2$-chain 
       \[ c_i:= [(0,i), (0,i+1), (1,i) ] + [ (1,i), (0,i+1), (1,i+1)]\]
       is the square 
       \[[(0,i), (0,i+1)] - [(0,i), (1,i)] - [(1,i),(1,i+1)] + [(0,i+1),(1,i+1)]. \]
       So the boundary of the $2$-chain $\sum_{i = 0}^{\ell} c_i$ is 
       $(0 \times C) - (1 \times C)$, and so its image $\phi(0 \times C) - \phi(1 \times C)$
       is also a boundary. 
       Thus $[\phi(0 \times C)] = [\phi(1 \times C)]$ in $H_1(\DH)$, as needed. 
    \end{proof}

\section{Reduction to Non-flat $H$-colouring}\label{sect:non-flat}

  In this section we reduce our main result to a problem called {\em non-flat $H$-colouring}.
  
  \begin{definition}\label{def:flat}
    A homomorphism $\phi: G \to H$ is {\em flat} if for all cycles $C$ in $G$, $[\phi(C)] = 0$ in $H_1(\DH)$. Equivalently $\phi$ is flat if $[\phi(C)] = 0$ for all cycles $C$ in a basis of $H_1(\DG)$. 
   \end{definition}

 Consider the following problem. 
\begin{problem*}[$\NF(H)$] \label{probflat}
\instance{A graph $G$.}
\decision{Is there a non-flat homomorphism of $G$ to $H$?} 
\end{problem*}
  
 We will prove the following.
 
 \begin{theorem}\label{thm:NF}
   Let $H$ be a reflexive graph such that $H_1(\DH)$ is free and nontrivial.
   The problem $\NF(H)$ is $\NP$-hard. 
 \end{theorem}
  
 The proof of this takes up most of the paper, but for the rest of this section, we show that it implies Theorem \ref{thm:full}. We need a result from \cite{LNS21} where homotopy classes of cycles in a graph are defined.

 \begin{definition}\label{def:moves}
 Let $\Pi(H;x_0)$ be the graph whose vertices are cycles in $H$ starting at $x_0$, and in which a cycle $Y$ is adjacent to a cycle $X = (x_0, \dots, x_\ell,x_0)$ if either of the following are true:
  \begin{enumerate}
  \item[\mylabel{P1}{(P1)}]
    $Y = (x_0,x_1, \dots, x_{i-1},x_i,x_i,x_{i+1},\dots, x_\ell,x_0)$ for any $i$, or
  \item[\mylabel{P2}{(P2)}]
    $Y = (x_0,x_1, \dots, x_{i-1},x'_i,x_{i+1}, \dots, x_\ell,x_0)$ for some $i \notin \{0,\ell\}$, where $x_{i-1}=x_{i+1}$.
  \end{enumerate}
  Let $\Pi(H)$ be the graph we get from the disjoint union, over all $x_0 \in V(H)$, of $\Pi(H;x_0)$ 
    by adding an edge between walks $y$ and $x = (x_0,x_1, \dots,x_{\ell-1},x_0)$ if 
      \begin{enumerate}
          \item[\mylabel{P3}{(P3)}] $y = (x_1, \dots, x_{\ell-1})$ and $x_1 = x_{\ell-1}.$
      \end{enumerate}
  \end{definition}
  
  Observe that if $H$ is connected, then all the constant cycles, those in which all vertices are the same, are in the same component of $\Pi(H)$. A cycle is {\em contractible} if it is in the same component of $\Pi(H)$ as the constant cycles. 
  
   The following is from \cite{LNS21}.  
   \begin{corollary}\label{cor:LNS}
    A homomorphism $\phi$ in $\bHom(G,H)$ is in the constant component of $\bHom(G,H)$ if 
    for every cycle $C$ in $G$, $\phi(C)$ is contractible in $\Pi(H)$. 
   \end{corollary}
     
    Now observe the following analog of a well known idea of topology.

    \begin{lemma}\label{lem:homhom}
        Let $H$ be a triangle-free reflexive graph. 
        If $[\phi(C)]$ is trivial in $H_1(\DH)$, then $\phi(C)$ is contractible in $\Pi(H)$. 
    \end{lemma} 
    \begin{proof}
      Our proof is by induction on the length of $\phi(C)$, it is clearly true of cycles of length $1$.  
      Let $\phi(C) = (x_0, \dots, x_\ell,x_0)$, for $\ell \geq 1$. 
      As a cycle in $Z_1(\DH)$ we write this as
      \[ [x_0,x_1] + [x_1,x_2] + \dots + [x_\ell,x_0]. \] 
      
      If the negation of some element in the chain appears later in the chain: $[x_j, x_{j+1}] = [x_{i+1}, x_i]$ for some $i < j$, then these disappear in $H_1(\DH)$, and we get in $H_1(\DH)$ that 
      $[ \phi(C) ] = [C_1] + [C_2]$ where
      \[ C_1 = [x_{i+1}, x_{i+2}] +  \dots + [x_{j-1}, x_j] \mbox{ and } C_2 = [x_{j+1}, x_{j+2}] + \dots +[x_{i-1}, x_i] \]
      with indices modulo $\ell$ (one of these might be empty).  The result therefore follows by induction.
      
      So we may assume that the $1$-chain $\phi(C) =  [x_0,x_1] + [x_1,x_2] + \dots + [x_\ell,x_0]$ consists of non-negated summands.   But as $Z_1(\DH)$ is free and $B_1(\DH)$ is as given in Example \ref{example:basprop}(2) the fact that $\phi(C) = 0$ means that all summands in this chain are $[x_0,x_0]$, and so $\phi(C)$
       is constant, and so contractible in $\Pi(H)$. 
    \end{proof}
    
     Thus from Corollary \ref{cor:LNS} we get the following.
     
     \begin{fact}
        Let $H$ be a triangle-free reflexive graph. If $\phi: G \to H$ is flat in $H_1(\DH)$, then $\phi$ is in the constant component of $\bHom(G,H)$.  If all homomorphisms $\phi: G \to H$ are flat, then $\bHom(G,H)$ is connected.
     \end{fact}
     
 With this we can show the following. 
 \begin{lemma}\label{lem:3toNF}
     If $\NF(H)$ is $\NP$-complete for a triangle-free reflexive cycle $H$, then $\Mix(H)$ is $\coNP$-complete.  
 \end{lemma}
 \begin{proof}
   We observe that the negation of $\NF(H)$ can be solved by $\Mix(H)$.  Indeed let $G$ be an instance of $\NF(H)$.  If it is a YES instance, then there is a non-flat homomorphism $\phi: G \to H$.  Thus for some cycle $B$ of $G$, $[\phi(B)] \neq 0$. By Lemma \ref{lem:homhom} we have then that $\phi$ is non-contractible in $\Pi(H)$, so $G$ is NO instance of $\Mix(H)$.
   
   On the other hand, if $G$ is a NO instance of $\NF(H)$ then all homomorphisms $\phi: G \to H$ are flat, and so by Lemma \ref{lem:homhom} and Corollary \ref{cor:LNS} reconfigure to constant homomorphisms.  Thus $G$ is a YES instance of $\Mix(H)$.   
 \end{proof}

  Observe now that Theorem \ref{thm:full} follows from Theorem \ref{thm:NF}.
 
  \begin{proof}[Proof of Theorem \ref{thm:full}]
    Let $H$ be a triangle-free reflexive graph containing a cycle.  By Example  \ref{example:basprop}(4) $H_1(\DH)$ is free, and as it contains a cycle, it is nontrivial,
    so we may apply Theorem \ref{thm:NF} to get that $\NF(H)$ is $\NP$-complete.  Lemma \ref{lem:3toNF} then tells us that $\Mix(H)$ is $\coNP$-complete. So Theorem \ref{thm:full} is proved.
  \end{proof}

    For the rest of the paper, we therefore concentrate on proving Theorem \ref{thm:NF}.

 \section{The homology sum gadget}\label{sect:windsum}

For an integer $n$ let $[n]$ denote the set $\{0,1, \dots, n-1\}$. An $n$-cycle will always have vertex set $[n]$, with consecutive vertices modulo $n$ being adjacent.  Let $\sH$ be a connected triangle-free reflexive graph with nontrivial homology, and let $\ZH$ be a particular nontrivial cycle of minimum girth. Let $g \geq 4$ be its girth, and let $Z$ be a generic $g$-cycle.  Throughout the rest of the paper, $\sH$, $\ZH$, $g$ and $Z$ will always have this setup.  

In constructions, we will often identify two copies of $Z$.  Unless we say otherwise, we do so identically-- identifying the copy of the vertex $i$ in one with the copy of $i$ in the other.

  Let $\sB$ be the symmetric reflexive cycle of girth $sg$. Observing that its vertex set can be written $V(\sB) = \{ ig + j \mid i \in [s], j \in [g] \}$, define the following maps, which are clearly in $\Hom(\sB,Z)$. Define  $\beta: \sB \to Z$ by  \[ \beta(ig + j ) = i, \] and for $k \in [s]$ define $\alpha_k: \sB \to Z$ by     \[ \alpha_k(ig + j ) = \PWF{j}{k = i}{0}{\mbox{ otherwise.}} \]   Intuitively, $\beta$ maps the first $s$ vertices of $\sB$ to $0$, the next $s$ vertices to $1$, etc.,  while $\alpha_k$ maps all vertices to $0$ except the $k^{th}$ run of $g$ vertices which it maps  around $Z$. Observe that there are paths from $[\beta]$ to $[\alpha_0]$, and from $[\alpha_0]$ to $[\alpha_1]$ to $\dots$ to $[\alpha_k]$ in $\bHom(\sB ,Z)$ in which the first vertex is always $0$. 
  In the case of $(s,g) = (3,4)$ for example, writing $f: \sB \to Z$ as $f(0)f(1), \dots, f(s\cdot g)$ it is easy to find such a path  
   \[000111222333 \sim 011111222333 \sim 012222222333 \sim 012333333333 \sim 012300000000 \]
  from $\beta$ to $\alpha_0$ and 
   \[012300000000 \sim 001230000000 \sim 000123000000 \sim \dots \sim 000001230000 \]
   from $\alpha_0$ to $\alpha_1$. 
  From this it is not too hard to see that the following is true.

  \begin{fact}\label{fact:lexists}
    There is an integer $\ell = \ell_s \geq 2$ such that for each $i \in [s]$ there
    is a homomorphism $\gamma_i : P_\ell \vtimes \sB \to Z$ such that
     \begin{enumerate}
      \item $\gamma_i|_{0 \times \sB} = \alpha_i$, and  
      \item $\gamma_i|_{\ell \times \sB} = \beta$, and 
      \item $\gamma_i$ is constant $0$ on $T:= P_\ell \times 0$.
      \end{enumerate}
    \end{fact}

    With this $\ell$ we are now ready to construct $S_s$. See Figure \ref{fig:windsum}.
    \begin{construction}\label{const:sum}
    Let $S_s$ be constructed from $P_\ell \vtimes \sB$, where $\ell = \ell_s$,
    as follows. 
      \begin{enumerate}
          \item In the last slice $\ell \times \sB$ of $P_\ell \vtimes \sB$, for each $j \in [g]$ identify the vertices in the set
             \[ \{ (\ell, i + js) \mid i \in [s]. \} \]
             This results in a copy of $\ZS$ of $Z$.   
          \item In the first slice $0 \times \sB$ of $P_\ell \vtimes \sB$, identify the vertices in the set 
              \[ \{ (0, ig) \mid i \in [s]. \} \]
              This creates $s$ copies of $Z$ sharing the vertex $(0,0)$. Refer to these as $\A_0', \dots, \A_{s-1}'$.
          \item For each $i \in [s]$, take a new copy of $P_1 \times Z$; identify one slice of it with $\A_i'$ and call the other slice $\A_i$.     
      \end{enumerate}
\end{construction}

 Let $S_s^{(0)}$ refer to $P_\ell \vtimes \sB$ and $S_s^{(i)}$ to the graph constructed in  step (i) of the construction.  The identifications in steps (1) and (2), can be viewed as homomorphisms $q_1: S_s^{(0)} \to S_s^{(1)}$ and $q_2: S_s^{(1)} \to S_s^{(2)}$. Notationally we should use names such as $q_2 \circ q_1( (\ell, 0) )$ for vertices of $ S_s^{(2)} \leq S_s^{(3)} = S_s$ but the abuse of calling this $(\ell, 0)$ is convenient and does not cause problems.

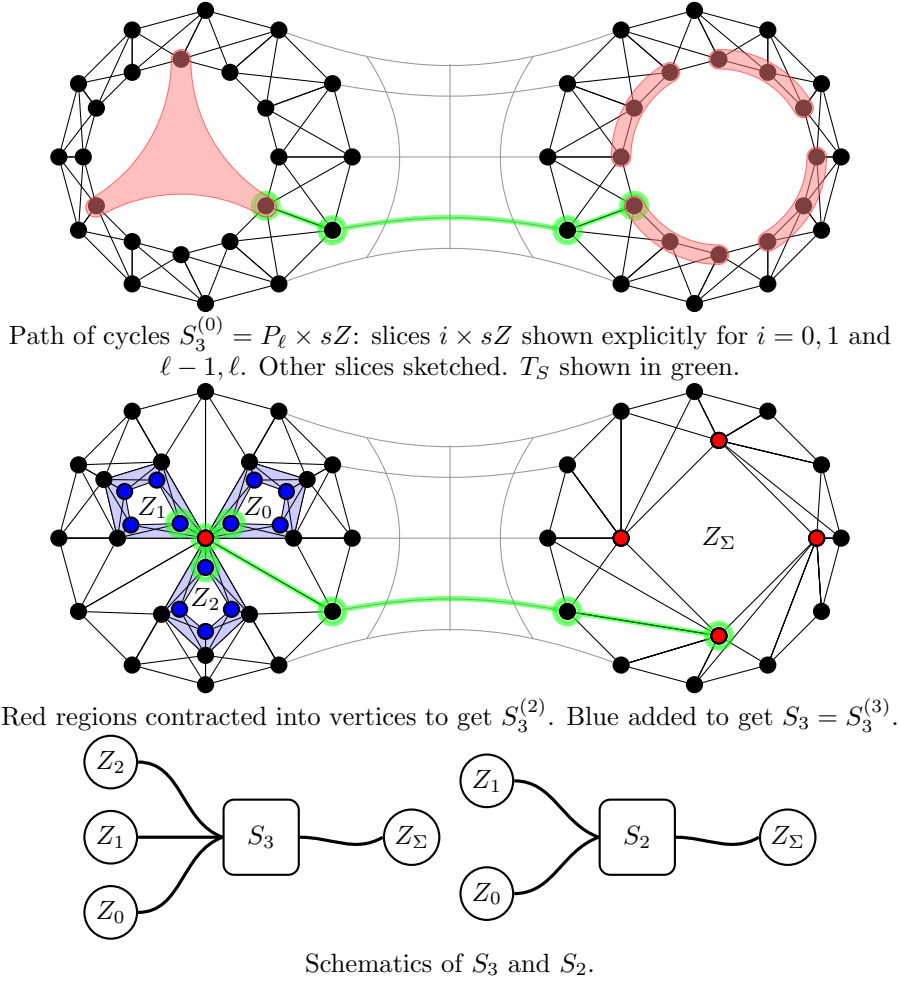
\begin{figure}     
\centering
 \begin{tikzpicture}[scale=.65, every node/.style = {vert}]
       \begin{pgfonlayer}{fg1}
        \foreach \i[evaluate=\i as \deg using {90 - (30*\i)},
                    evaluate=\i as \ri using {int(Mod(-\i,12))}] in {0,...,11}{
          \draw (\deg:2cm) + (-5.5,0) node (1\i) {};
          \draw (\deg:3cm) + (-5,0) node (2\i) {};
          \draw (\deg:2cm) + (5.5,0) node (7\ri) {};
          \draw (\deg:3cm) + (5,0) node (6\ri) {};     }  
        \end{pgfonlayer}
        \foreach \i[evaluate=\i as \next using {int(Mod(\i+1,12))}] in {0,...,11}{
            \vsquare[none]{1}{2}{\i}{\next}
           \vsquare[none]{7}{6}{\i}{\next} }
        \foreach \i/\b in {1/20,2/10,3/0,4/-10,5/-20}{\draw[gray!80] (2\i) edge[bend right = \b] (6\i);}
        \foreach \i/\j in {-1.7/2.05,0/1.9,1.7/2.05}{\draw[gray!80] (\i,\j) edge[bend right = 20*\i] (\i,-\j);}       
       \begin{pgfonlayer}{fg2}
         \begin{scope}[xshift = -5.5cm]
           \path[draw, red, use Hobby shortcut, fill=red!50, opacity=.5] (-.2,2) arc (180:0:.2) arc (180:240:3.3) arc (60:-120:.2) arc (60:120:3.3) arc(300:120:.2) arc (-60:0:3.3) -- cycle;
         \end{scope}
         \begin{scope}[xshift = 5.5cm]
            \path[draw, red, use Hobby shortcut, fill=red!50, opacity=.5] (0,1.8) arc (270:90:.2) arc (90:30:2.2) arc (30:-150:.2) arc (30:90:1.8) -- cycle;
           \path[draw, red, use Hobby shortcut, fill=red!50, opacity=.5] (1.8,0) arc (180:0:.2) arc (0:-60:2.2) arc (300:120:.2) arc (-60:0:1.8) -- cycle; 
            \path[draw, red, use Hobby shortcut, fill=red!50, opacity=.5] (0,-1.8) arc (90:-90:.2) arc (270:210:2.2) arc (210:30:.2) arc (210:270:1.8) -- cycle; 
           \path[draw, red, use Hobby shortcut, fill=red!50, opacity=.5] (-1.8,0) arc (360:180:.2) arc (180:120:2.2) arc (120:-60:.2) arc (120:180:1.8) -- cycle;  
         \end{scope}
       \end{pgfonlayer} 

         \begin{scope}[color=green, line width = 2pt, opacity=.5, node distance = 0em, every node/.style = {circle, minimum size = .3cm, draw = green, fill=green!50, opacity=.5}]
         \path (14) node {} edge (24); \path (24) node {}  edge[bend right = -10] (64); \path (64) node {} edge (74) ; \path (74) node {}; \
         \end{scope}
         
    \end{tikzpicture}

     Path of cycles $S_3^{(0)}=P_\ell \vtimes \sB$: slices $i \times \sB$ shown explicitly for $i = 0,1$ and $\ell-1,\ell$. Other slices sketched. $T_S$ shown in green.

     \begin{tikzpicture}[scale=.65, every node/.style = {vert}]  
   \begin{pgfonlayer}{fg1}
         \foreach \i[evaluate=\i as \deg using {90 - (30*\i)},
                     evaluate=\i as \ri using {int(Mod(-\i,12))}] in {0,...,11}{
           \draw (\deg:3cm) + (-5,0) node (2\i) {};
           \draw (\deg:3cm) + (5,0) node (6\ri) {};
         }  
          \foreach \i[evaluate=\i as \deg using {90 - (90*int(floor(Mod(-\i,12)/3))))}] in {0,...,11}{
           \draw (\deg:2cm) + (5.5,0) node[fill=red] (7\i) {};
         }  
     
        \draw (0,0) + (-5,0) node[fill=red] (10) {};
        \draw ({1.8*cos(60)},{1.8*sin(60)}) + (-5,0) node (11) {};
        \draw ({2.4*cos(30)},{2.4*sin(30)}) + (-5,0) node (12) {};
        \draw ({1.8*cos(0)},{1.8*sin(0)}) + (-5,0) node (13) {};
        \draw (0,0) + (-5,0) node[fill=red] (14) {};
        \draw ({1.8*cos(-60)},{1.8*sin(-60)}) + (-5,0) node (15) {};
        \draw ({2.4*cos(-90)},{2.4*sin(-90)}) + (-5,0) node (16) {};
        \draw ({1.8*cos(-120)},{1.8*sin(-120)}) + (-5,0) node (17) {};
        \draw (0,0) + (-5,0) node[fill=red] (18) {};
        \draw ({1.8*cos(180)},{1.8*sin(180)}) + (-5,0) node (19) {};
        \draw ({2.4*cos(150)},{2.4*sin(150)}) + (-5,0) node (110) {};
        \draw ({1.8*cos(120)},{1.8*sin(120)}) + (-5,0) node (111) {};

        \draw ({.6*cos(30)},{.6*sin(30)}) + (-5,0) node[fill=blue] (00) {};
        \draw ({1.55*cos(60-10)},{1.55*sin(60-10)}) + (-5,0) node[fill=blue] (01) {};
        \draw ({1.91*cos(30)},{1.91*sin(30)}) + (-5,0) node[fill=blue] (02) {};
        \draw ({1.55*cos(0+10)},{1.55*sin(0+10)}) + (-5,0) node[fill=blue] (03) {};
        \draw ({.6*cos(-90)},{.6*sin(-90)}) + (-5,0) node[fill=blue] (04) {};
        \draw ({1.55*cos(-60-10)},{1.55*sin(-60-10)}) + (-5,0) node[fill=blue] (05) {};
        \draw ({1.91*cos(-90)},{1.91*sin(-90)}) + (-5,0) node[fill=blue] (06) {};
        \draw ({1.55*cos(-120+10)},{1.55*sin(-120+10)}) + (-5,0) node[fill=blue] (07) {};
        \draw ({.6*cos(150)},{.6*sin(150)}) + (-5,0) node[fill=blue] (08) {};
        \draw ({1.55*cos(180-10)},{1.55*sin(180-10)}) + (-5,0) node[fill=blue] (09) {};
        \draw ({1.91*cos(150)},{1.91*sin(150)}) + (-5,0) node[fill=blue] (010) {};
        \draw ({1.55*cos(120+10)},{1.55*sin(120+10)}) + (-5,0) node[fill=blue] (011) {}; 

        \draw (30:1.25) + (-5,0) node[draw=none, fill=none] (){$\A_0$};
        \draw (150:1.25) + (-5,0) node[draw=none, fill=none] (){$\A_1$};
        \draw (270:1.25) + (-5,0) node[draw=none, fill=none] (){$\A_2$};
        \draw (0:0) + (5.5,0) node[draw=none, fill=none] (){$\ZS$};
         \end{pgfonlayer}
        
        \foreach \i[evaluate=\i as \next using {int(Mod(\i+1,12))}] in {0,...,11}{
            \vsquare[none]{1}{2}{\i}{\next}
            \vsquare[none]{7}{6}{\i}{\next} }
        
        \foreach \a/\b in {0/1, 1/2, 2/3, 3/0, 4/5, 5/6, 6/7, 7/4, 8/9, 9/10, 10/11, 11/8}{
        \vsquare[blue!20]{0}{1}{\a}{\b}}

        \foreach \i/\b in {1/20,2/10,3/0,4/-10,5/-20}{\draw[gray!80] (2\i) edge[bend right = \b] (6\i);}  \foreach \i/\j in {-1.7/2.05,0/1.9,1.7/2.05}{\draw[gray!80] (\i,\j) edge[bend right = 20*\i] (\i,-\j);}
        
         \begin{scope}[color=green, line width = 2pt, opacity=.5, node distance = 0em, every node/.style = {circle, minimum size = .3cm, draw = green, fill=green!50, opacity=.5}]
         \path (14) node {} edge (24); \path (24) node {}  edge[bend right = -10] (64); \path (64) node {} edge (74) ; \path (74) node {}; 
         \path (04) node {} edge (14);       \path (08) node {} edge (14);       \path (00) node {} edge (14);  

     \end{scope}
        
\end{tikzpicture}

Red regions contracted into vertices to get $S_3^{(2)}$. Blue added to get $S_3 = S_3^{(3)}$.

 \begin{tikzpicture}[every node/.style = {vert}] 
   
    \draw (2,0) node[cycle, solid] (Z) {$\ZS$};
    \draw (0,0) node[square] (Sq) {$S_3$};
    \foreach \i/\y in {0/-1,1/0,2/1}{
          \draw (-2,\y) node[cycle, solid] (A\i) {$\A_{\i}$}; }  
     \draw (Sq.east) edge[very thick, out=0,in=-150  ] (Z.west);
     \draw (Sq.west) edge[very thick, out=200, in=0] (A0.east);
     \draw (Sq.west) edge[very thick, out=180, in=0] (A1.east);
     \draw (Sq.west) edge[very thick, out=160, in=0] (A2.east);

      \begin{scope}[xshift=5cm]
    \draw (2,0) node[cycle, solid] (Z) {$\ZS$};
    \draw (0,0) node[square] (Sq) {$S_2$};
    \foreach \i/\y in {0/-.75,1/.75}{
          \draw (-2,\y) node[cycle, solid] (A\i) {$\A_{\i}$}; }  
     \draw (Sq.east) edge[very thick, out=0,in=-150  ] (Z.west);
     \draw (Sq.west) edge[very thick, out=200, in=0] (A0.east);
     \draw (Sq.west) edge[very thick, out=160, in=0] (A1.east);

      \end{scope}
      
     \end{tikzpicture}

     Schematics of $S_3$ and $S_2$.
 
      \caption{Construction \ref{const:sum} of the sum gadget $S = S_3$ with $g = 4$. $T_S$ overset in green.}
      \label{fig:windsum}
 \end{figure}

 This allows us the following definition that we will need later. 
 
 \begin{definition}
   A cycle in $S_s$ is a {\em gadget-cycle} if it is 
   a slice $i \times \sB$ of $S_s^{(0)}$ or is $\A_i$ for some $i$.  The cycles $\A_i$ and $\ZS$ are {\em end gadget-cycles}.
 \end{definition}
 
 Note that the cycles $\A'_0, \dots, \A'_s$ are considered as a single gadget-cycle.
 With this definition we prove the following desired properties of $S_s$.  

 For intuition, recall that in the case that $H$ is simply the cycle $\ZH$, the value $[\Gamma(C)]$ really just counts, with orientation, the number of times that $\Gamma$ winds a cycle $C$ around the target $\ZH$.  The end gadget-cycles are copies of $Z$ so under a mapping can only wind $-1,0$ or $1$ times around $\ZH$. Part (2) of the following lemma then says that the winds of the $\A_i$ sum to the wind of $\ZS$.  
  
 \begin{lemma}\label{gad:sum}
   The graph $S = S_s$ contains end gadget-cycle copies $\ZS ,\A_0, \dots, \A_{s-2}$ and $\A_{s-1}$ of $Z$, which are pairwise distance at least $2$ apart;  and a tree $T = T_S$ that
   contains the copy of $0$ in each of these end gadget-cycles. 
   Furthermore,  the following are true.   
   \begin{enumerate}
    \item
       For any homomorphism $\phi: S \to \sH$, if $[\phi(\ZS)] = 0$, then $[\phi(C)]=0$ for all gadget-cycles $C$ of $S$ except possibly the $\A_i$. 
   \item For any homomorphism $\phi :S \to \sH$, we have 
     \[ [ \phi(\ZS) ] = \sum_{i \in [s]} [\phi(\A_i)]. \] 
   \item For each $i \in [s]$ there is a homomorphism $\Gamma_i: S \to \ZH \leq \sH$
     such that $[\Gamma_i(\A_i)] = [\ZH]$ and such that $\Gamma_i$ is $0$ on
     $T$ and on $\A_j$ for $j \neq i$.   
   
       \end{enumerate}
     \end{lemma}
     \begin{proof}
       The statements before the `Further' are just recollections of the construction. 
       
      For part (1), let $\phi: S \to \sH$ be a homomorphism with $[\phi(\ZS)] = 0$. So $\phi' = \phi \circ (q_2 \circ q_1)$ has $[\phi'(\ell \times \sB)] = 0$. By Fact \ref{fact:preshom} we thus get $[\phi'(i \times \sB)] = 0$ for all $i \in [\ell]$, and thus $[\phi(i \times \sB)] = [\phi(q_2 \circ q_1(i \times \sB))] = 0$, as needed.
      
       For part (2), let $\phi$ be any homomorphism of $S$ to $\sH$. 
       Where $S_s^{(0)}$ is as in Construction \ref{const:sum}, we have
       $\ZS = (q_2 \circ q_1)(\ell \times \sB)$ and $(q_2 \circ q_1)(0 \times \sB) = \sum_{i \in [s]}\A'_i$, so 
       the homomorphism $\phi' = \phi \circ (q_2 \circ q_1): S_s^{(0)} \to S_s^{(2)} \to H$ satisfies
          \[  [\phi(\ZS)] = [\phi'( \ell \times \sB )] = [\phi'(0 \times \sB)] = \sum [\phi(\A'_i)], \]
        where the middle equality is from Fact \ref{fact:preshom}. Using Fact \ref{fact:preshom} again on the copies of $P_1 \times Z$ connecting the  $\A_i$ and $\A_i'$ we get (2).

      Now let $T$ be the tree in $S_s$ induced by the vertices $(i, 0)$ in $S^{(2)}_s$ and the vertex $0$ in $\A_i$ for each $i$.
      To get part (3), observe that the homomorphism $\gamma_i: S_s^{(0)} \to Z$ from Fact \ref{fact:lexists} maps vertices identified by $(q_2 \circ q_1)$ to the same images, so
      factors as $\gamma_i = (q_2 \circ q_1) \circ \Gamma'_i$ for some $\Gamma'_i: S_s^{(2)} \to Z$ that is $0$ on $T \cap S^{(2)}$ and on $\A'_j$ for all $j \neq i$.
      Extend $\Gamma'_i$ to $\Gamma_i: S_s \to Z$ by letting it act on $\A_j$ as is does on $\A'_j$ for all $j \in [s]$; and then map it identically from $Z$ to $\ZH$.

      \end{proof}

  \section{The main construction }\label{sect:stage1}

Having built our gadget $S_s$, we now give our main construction, that of $\Gb$.
  This construction was outlined in Section \ref{sect:outline}, and most of it is illustrated in Figure \ref{fig:mainconst}.  In Propostion \ref{prop:reduct1}
we give most of the desired properties, and then use it to prove Theorem \ref{thm:NF} modulo Proposition \ref{prop:topo} which we prove in the next section. 

  \begin{construction}\label{const:G*}
  Recall that $Z$ is a cycle of girth at least $4$ isomorphic to the the shortest cycle in $H$.  
  For an irreflexive graph $G$ construct $\Gb$ as follows.
  \begin{enumerate}
  \item For every vertex $v$ of $G$ let $\Gb$ contain three copies $\A^v_0, \A^v_1$ and
    $\A^v_2$ of $Z$. Let $\Gb$ also contain a copy $\sZ$ of $Z$. Let all of these be
    disjoint. 
  \item For every vertex $v$ of $G$ create a new copy $S^v$ of the sum gadget $S_3$;
      identify the copy of $\ZS$ in $S^v$ with $\sZ$, and for each $i \in [3]$ identify the copy of $\A_i$ in $S^v$ with $\A^v_i$ .
  \item For every vertex $v$ of $G$ and every pair of distinct indices $i < j \in [3]$,
    create a new copy $S^v_{i,j}$ of $S_2$; identify its copy of $\A_0$ with $\A^v_i$ and its
    copy of $\A_1$ with $\A^v_j$. However, in identifying this copy of $\A_1$ with $\A^v_j$
    identify the copy of the vertex $i$ in $\A_1$ with the vertex $g-i$ in $\A^v_j$.
   \item For each edge $uv$ in $G$ and each $i \in [3]$, create a new copy $S_i^{uv}$ of $S_2$;     identify its copy of $\A_0$ with $\A_i^u$ and its copy of $\A_1$ with $\A_i^v$.  \item  Let $\sT$ be the union of the trees $T_S$ in all the gadgets $S$ used in the previous steps.  For each cycle $C$ in a basis of $H_1(\bold{T_S})$ add a new vertex $v_C$ with edges to every vertex of $C$. 
  \end{enumerate}
  Let $\Ga[i]$ refer to the graph constructed in step (i) of the construction. 
  \end{construction}

\begin{figure}
\captionsetup{singlelinecheck=off}

\caption*{Step $(2)$.}
     
\begin{tikzpicture}[every node/.style = {vert}]
     \foreach \v/\x/\y in {a/0/0, b/3/0, c/6/0}{
       \draw (\x,\y) node[clump] (c\v) {}; }
     \foreach \v in {a, b, c}{
       \foreach \i/\a in {0/90,1/210,2/330}{
          \draw (c\v) + (\a:.5) node[cycle, solid] (A\v\i) {$\A_{\i}^{\v}$}; }}  
    \draw (9,2) node[cycle, solid] (Z) {$Z_*$};

    
     \draw (0,2) node[square] (Sa) {$S^a$};
      
     \draw (Sa.east) edge[very thick, out=30,in=150  ] (Z.west);
     \draw (Sa.south) edge[very thick, out=-90, in=70] (Aa0.north);
     \draw (Sa.south) edge[very thick, out=-140, in=110] (Aa1.north);
     \draw (Sa.south) edge[very thick, out=-30, in=70] (Aa2.north);

     \draw (3,2) node[square] (Sb) {$S^b$};
     \draw (Sb.east) edge[very thick, out=30,in=150  ] (Z.west);
     \draw (Sb.south) edge[very thick, out=-90, in=70] (Ab0.north);
     \draw (Sb.south) edge[very thick, out=-140, in=110] (Ab1.north);
     \draw (Sb.south) edge[very thick, out=-30, in=70] (Ab2.north);

     \draw (6,2) node[square] (Sc) {$S^c$};
     \draw (Sc.east) edge[very thick, out=-30,in=150  ] (Z.west);
     \draw (Sc.south) edge[very thick, out=-90, in=70] (Ac0.north);
     \draw (Sc.south) edge[very thick, out=-140, in=110] (Ac1.north);
     \draw (Sc.south) edge[very thick, out=-30, in=70] (Ac2.north);

\end{tikzpicture}
\caption*{Step $(3)$.}

\begin{tikzpicture}[every node/.style = {vert}]
     \foreach \v/\x/\y in {v/0/0}{
       \draw (\x,\y) node[clump] (c\v) {}; }
     \foreach \v in {v}{
       \foreach \i/\a in {0/90,1/210,2/330}{
          \draw (c\v) + (\a:.5) node[cycle, solid] (A\v\i) {$\A_{\i}^{\v}$}; }}  

     \draw (160:3) node[square] (S01) {$S^v_{0,1}$};
     \draw (S01) + (180:1.5)  node[cycle, solid] (S01Z){}; 
     \draw (S01.west) edge[very thick, out=180, in=-20]  (S01Z.east);
     \draw (S01.east) edge[very thick, dashed, out=-30, in=140] (Av1.west);
     \draw (S01.east) edge[very thick, out=0, in=180] (Av0.west);

     \draw (-160:3) node[square] (S12) {$S^v_{1,2}$};
     \draw (S12) + (180:1.5)  node[cycle, solid] (S12Z){}; 
     \draw (S12.west) edge[very thick, out=180, in=20]  (S12Z.east);
     \draw (S12.east) edge[very thick,  out=20, in=-140] (Av1.west);
     \draw (S12.east) edge[very thick, dashed, out=-30, in=-130] (Av2.south);

     \draw (0:3) node[square] (S20) {$S^v_{0,2}$};
     \draw (S20) + (0:1.5)  node[cycle, solid] (S20Z){}; 
     \draw (S20.east) edge[very thick, out=0, in=-150]  (S20Z.west);
     \draw (S20.west) edge[very thick, dashed, out=190, in=20] (Av2.east);
     \draw (S20.west) edge[very thick,  out=150, in=-20] (Av0.east);

\end{tikzpicture}
\caption*{For every vertex $v$ of $G$. Dashed lines on the gadgets mean that the copy of $Z$ at the end is identified with the labelled copy in the reverse direction.}
\caption*{Step $(4)$.}

\begin{tikzpicture}[every node/.style = {vert}]
     \foreach \v/\x/\y in {a/0/0, b/3/0, c/6/0}{
       \draw (\x,\y) node[clump] (c\v) {}; }
     \foreach \v in {a, b, c}{
       \foreach \i/\a in {0/90,1/210,2/330}{
          \draw (c\v) + (\a:.5) node[cycle, solid] (A\v\i) {$\A_{\i}^{\v}$}; }}  


     \foreach \a/\b/\x in {a/b/1.5, b/c/4.5}{
      \draw (\x,2) node[square] (S0\a\b) {$S_0^{\a\b}$};
      \draw (S0\a\b) + (90:1) node[cycle, solid] (S0\a{\b}Z){};
      \draw (S0\a\b.north) edge[very thick, out=90, in=-90] (S0\a{\b}Z.south);
      \draw (S0\a\b.south) edge[very thick, out=-110, in=90] (A\a0.north);
      \draw (S0\a\b.south) edge[very thick, out=-70, in=90] (A\b0.north);}

     \foreach \a/\b/\x in {a/b/1, b/c/4}{
      \draw (\x,-2) node[square] (S1\a\b) {$S_1^{\a\b}$};
      \draw (S1\a\b) + (-90:1) node[cycle, solid] (S1\a{\b}Z){};
      \draw (S1\a\b.south) edge[very thick, out=-90, in=90] (S1\a{\b}Z.north);
      \draw (S1\a\b.north) edge[very thick, out=110, in=-90] (A\a1.south);
      \draw (S1\a\b.north) edge[very thick, out=70, in=-90] (A\b1.south);}

     \foreach \a/\b/\x in {a/b/2.5, b/c/5.5}{
      \draw (\x,-2) node[square] (S2\a\b) {$S_2^{\a\b}$};
      \draw (S2\a\b) + (-90:1) node[cycle, solid] (S2\a{\b}Z){};
      \draw (S2\a\b.south) edge[very thick, out=-90, in=90] (S2\a{\b}Z.north);
      \draw (S2\a\b.north) edge[very thick, out=140, in=-50] (A\a2.south);
      \draw (S2\a\b.north) edge[very thick, out=30, in=-130] (A\b2.south);}
\end{tikzpicture}

\caption{Steps (2) - (4) of Construction \ref{const:G*} making $G_*$ when $G$ is the path $a,b,c$.  The three figures show what is added in the three steps of the construction. Their union is $\Ga[4]$}\label{fig:mainconst} 
  
\end{figure}

    Observe that the subgraph $\sT$ of $\Gb$ is no longer a tree, but is still connected, as all identifications of gadget-cycles identify the vertices $0$. 
     A cycle of $\Gb$ is a {\em gadget-cycle} if it is a gadget-cycle in one of the gadgets $S$ used in constructing $\Gb$.

  We now prove the main properties of the construction. Recall that $\sH$ is a graph with nontrivial homotopy and $\ZH$ is a 
  nontrivial cycle of $\sH$ of minimum girth.

  \begin{lemma}\label{lem:exactly1}
    For any homomorphism $\phi^*:\Gb \to \sH$ and any $v \in V(G)$, two of the
    values $[\phi^*(\A_i^v)]$ for $i \in [3]$ are $0$, and the other is $[\phi^*(\sZ)]$ (which may also be $0$).
    \end{lemma}
  \begin{proof}
     As the copy $S^v$ of $S_3$ in $\Gb$ connects $\sZ$ with the $\A_i^v$,  we have by (2) of Lemma \ref{gad:sum} that 
          \[ \sum_{i = 0}^2  [\phi^*(\A_i^v)]  = [\phi^*(\sZ)]. \]
     But as $\sZ$ and the $\A_i^v$ are just copies of the minimum  cycle of $\sH$, their images under $\phi^*$ are either $0$, or are generators of $H_1(\DsH)$.
     As $[Z_0]$ is a generator, no two of the $[\phi^*(\A_i^v)]$ can sum to it unless one of them is $0$.  The $S_2$ gadgets introduced in step (3) of Construction \ref{const:G*} imply that no two of the $[\phi^*(\A_i^v)]$ can sum to $0$ unless at least one of them is $0$. So the only possibility left is that two of the $[\phi^*(\A_i^v)]$ are $0$, and the last is $[\phi^*(\sZ)]$, as needed. 

  \end{proof}

  \begin{proposition}\label{prop:reduct1}
    Let $G$ be an irreflexive graph and $\Gb$ be constructed from $G$ by Construction \ref{const:G*}. 
       The following are equivalent.  
  \begin{enumerate}
   \item   There is a  homomorphism  $\phi^*: \Gb \to \sH$ such that $[\phi^*(C)] \neq 0$ for some gadget-cycle $C$ of $\Gb$. 
    \item There is a  homomorphism  $\phi^*: \Gb \to \sH$ such that $[\phi^*(\sZ)] \neq 0$.
   \item There  is a homomorphism $\phi: G \to K_3$.
   \end{enumerate}
  \end{proposition}

  \begin{proof}
   The implication (2) $\Rightarrow$ (1) is clear.  \\

   \noindent {\bf (1) $\Rightarrow$ (2).}  The contra-positive is clear from Lemma \ref{lem:exactly1}.  Indeed, assume that all homomorphisms
    $\phi^*: \Gb \to \sH$ have $\phi^*(\sZ) = 0$. Taking one, Lemma  \ref{lem:exactly1} tells us that $\phi^*(\A^v_i) = 0$ for all end gadget-cycles $\A^v_i$
    in $\Ga[2]$.
    The only other end gadget-cycles in $\Gb$ are those introduced in copies $S$ of $S^2$ in step (3) and (4) of the construction, but for each of these, 
    two of the end gadget-cycles are identified with some $\A^v_i$, so by Lemma \ref{gad:sum} the last end gadget-cycle, $\ZS$, must also have $\phi^*(\ZS) = 0$. 
    For each gadget $S$ used in the construction of $\Gb$ we have shown that $\phi^*(C) = 0$ for all end gadget-cycles $C$, so by part (1) of  Lemma \ref{gad:sum}, this holds for all gadget-cycles $C$ as well.  
    If all homomorphisms have $\phi^*(\sZ) = 0$ then they all have $\phi^*(C) = 0$ on all gadget-cycles.  This gives the claim. \\

     \noindent {\bf (2) $\Rightarrow$ (3).} Assume that there is a homomorphism $\phi^*:\Gb \to \sH$ with
    $[\phi^*(\sZ)] = [Z]$ for some nontrivial  $g$-cycle $Z$ in $\sH$.  For each vertex $v$, Lemma \ref{lem:exactly1} says that $[\phi^*(\A^v_i)] = [Z]$ for exactly one value of $i \in [3]$.  Define $\phi:G \to V(K_3)$ by setting $\phi(v)$ to be this one value. 
     For any edge $uv$ of $G$ we have $\phi(u) \neq \phi(v)$,
    as if $\phi(u) = i = \phi(v)$ then  $\phi^*$ restricted to
    $S_i^{uv} \cong S_2$ has
     $[\phi^*(\A_0)] = [\phi^*(\A_1)] = [Z]$, but then $[\phi^*(\ZS)] = 2[Z]$ by  Lemma \ref{gad:sum}; but this is impossible, as $\ZS$ has girth $g$ so $[\phi^*(\ZS)]$ is $0$ or a generator in $H_1(\DH)$.   Thus $\phi: G \to K_3$ is a homomorphism.  \\

\noindent {\bf  (3) $\Rightarrow$ (2).} Assume that there is a homomorphism $\phi: G \to K_3$. 
We construct a homomorphism $\phi^*: \Ga[1] \to \ZH$, and then extend it to $\Ga[i]$ for $i$ going from $2$ to $5$.  
Define  $\phi^*: \Ga[1] \to \ZH$ as follows. Let $\phi^*$ take $\sZ$ identically to $\ZH$. For each $v$ let  $\phi^*$ take $\A_v^i$ identically to $\ZH$ if $\phi(v) = i$ and otherwise take $\A_v^i$ constantly to $0$. This is clearly a homomorphism as each of these cycles are disjoint and independent.  

 For steps $i = 2$ to $4$, we get $\Ga[i]$ from $\Ga[i-1]$ by introducing new copies of $S_2$ or $S_3$ and identifying their end gadget-cycles to cycles in $\Ga[1]$.
 So in extending $\phi^*$ to these gadgets $S$ it is enough to check that the homomorphisms $\Gamma_i$ from Lemma \ref{gad:sum} which we extend them by agree with $\phi^*$ on the cycles of $\Ga[1]$.   Recall that $\Gamma_{i}$ maps $\ZS$ and $\A_i$ identically to $\ZH$, and maps the other $\A_j$ constantly to $0$.  

 To extend  $\phi^*$ to $\Ga[2]$, let $\phi^*$ be $\Gamma_{\phi(v)}$ on $S^v$ for each vertex $v$ of $G$.  
 To extend this to $\Ga[3]$, for each vertex $u$ of $G$, let the copy $S^u_{i,j}$, for $i < j$, be coloured with the colouring $\Gamma_0$ if $\phi(v) = i$, with colouring
    $\Gamma_1$ if $\phi(v) = j$, and with the constant $0$ colouring otherwise.
  To extend $\phi^*$ to $\Ga[4]$, for each edge $uv$ of $G$ and $i \in [3]$, at most one of $\phi(u)$ and $\phi(v)$
    is $i$.   Let $\phi^*$ be $\Gamma_0$ on $S_i^{uv}$ if
    $\phi(u) = i$,  be $\Gamma_1$ if $\phi(v) = i$, or
    be constant $0$ if neither of these holds.

    All of these definitions agree with $\phi^*$ on the cycles in $\Ga[1]$, so this homomorphism $\Ga[4] \to \ZH$ is well defined. Moreover, on each copy $S$ of the gadgets $S_2$ and $S_3$ in $\Ga[4]$ we have defined $\phi^*$ by one of the homomorphisms $\Gamma_i$ from Lemma \ref{gad:sum}, and these homomorphisms were all constant $0$ on the subgraph $T_S$ of $S$.  So $\phi^*$ is constant $0$ on $T_*$.  Thus we can extend $\phi^1$ to $\Ga[5] = \Ga$ by setting $\phi(v_C) = 0$ for each new vertex $v_C$ introduced in step (5) of the construction.    
  \end{proof}

     To prove Theorem \ref{thm:NF}, we need the following topological fact about our  construction.  

     \begin{proposition}\label{prop:topo}
       The group $H_1(\DG_*)$ has a basis  consisting of end gadget-cycles. 
     \end{proposition}
   
     Hopefully this result is intuitive, but its proof takes considerable work, so we postpone it until the next section.  Before that, we show that the above two propositions 
     give us Theorem \ref{thm:NF}.

      \begin{proof}[Proof of Theorem \ref{thm:NF}] 
        For an irreflexive instance $G$ of $3$-colouring, let $\Gb$ be the graph from Construction \ref{const:G*}. We show that $G$ is $3$-colourable if and only if $\Gb$ has a non-flat homomorphism to $\sH$.  As $3$-colouring is well known to be $\NP$-complete, and $|\Gb| \leq s|\sH|$ where $s$ is a constant depending only on $|S_2|$ and $|S_3|$, so only on the girth of $H$, this shows that $\NF(H)$ is $\NP$-hard.  
        
        Indeed, assume there is a non-flat $\sH$-colouring $\phi$ of $\Gb$. By Proposition \ref{prop:topo} we many assume that $[\phi(C)] \neq 0$ for some end gadget-cycle of $C$ of $\Gb$, so by Proposition \ref{prop:reduct1} there is a $3$-colouring of $G$.
        
        On the other hand, assume that there is a $3$-colouring of $G$. By Proposition \ref{prop:reduct1} there is a homomorphism $\phi: \Ga \to H$ with $\phi(\sZ) \neq 0$; this is non-flat.  
      \end{proof}

        Modulo the proofs of Proposition \ref{prop:topo}, which we will prove next, this completes the proof of Theorem \ref{thm:NF} and so,   as shown in Section \ref{sect:non-flat}, of Theorem \ref{thm:full}.

    \section{Computing the basis of $\Hom_1(\bold{G_*})$}\label{sect:NTbasis}

In this section we prove Proposition \ref{prop:topo}, showing that $H_1(\bold{\Gb})$ is generated by end gadget-cycles. 
In the first subsection we recall standard tools of homology, and prove lemmas that tailor them to our construction.  In the second subsection we will use these lemmas in our main proof.

    \subsection{Topological Tools}\label{sect:Kunneth}

  As the path-of-cycles is the main building block of our construction, we need to know its homology. 
    The notation $\bbP_\ell \times \DsC$ is ambiguous as the product space of the clique complexes $\bbP_\ell$ and $\DsC$ is not the same as clique-complex of the product $P_\ell \times \ZH$. Indeed the product $P_\ell \times C$ is shown on the left of Figure \ref{fig:products} as a tiling of a cylinder with copies of $K_4$, while the product space $\bbP_\ell \times \DsC$ (see \cite{Hatcher} pp. 277-278), shown on the right, is a triangulation of a cylinder.

    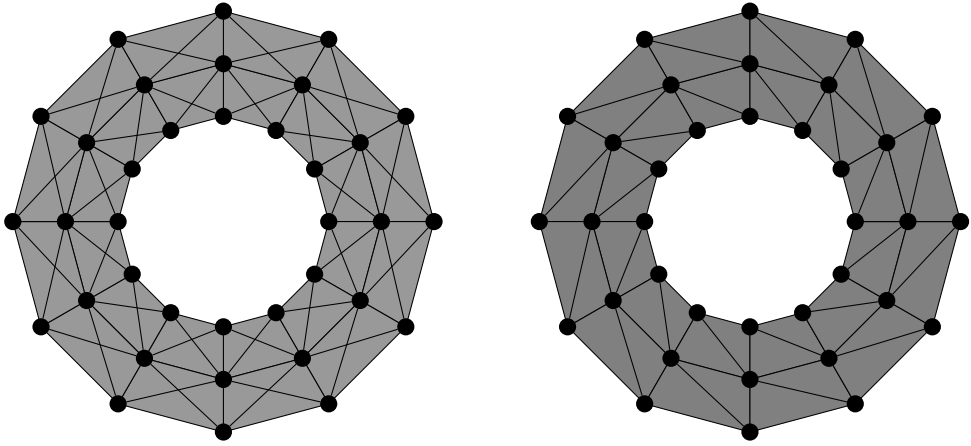
\begin{figure}
        \centering
        
\begin{tikzpicture}[scale=.7, every node/.style = {vert}]
     \begin{pgfonlayer}{fg1}
        \foreach \i[evaluate=\i as \deg using {90 - (30*\i)},
                    evaluate=\i as \ri using {int(Mod(-\i,12))}] in {0,...,11}{
          \draw (\deg:2cm) + (-5,0) node (1\i) {};
          \draw (\deg:3cm) + (-5,0) node (2\i) {};
          \draw (\deg:4cm) + (-5,0) node (3\i) {};
          \draw (\deg:2cm) + (5,0) node (7\ri) {};
          \draw (\deg:3cm) + (5,0) node (6\ri) {};
          \draw (\deg:4cm) + (5,0) node (5\ri) {};
        }  
        \end{pgfonlayer}
        
        \foreach \i[evaluate=\i as \next using {int(Mod(\i+1,12))}] in {0,...,11}{
            \vsquare{1}{2}{\i}{\next}
            \vsquare{2}{3}{\i}{\next} 
            \vhsquare{7}{6}{\i}{\next}
            \vhsquare{6}{5}{\i}{\next}}
         
\end{tikzpicture}  

        \caption{Products of a path and a cycle: on the left is the graph product, on the right is the product as simplicial complexes}
        \label{fig:products}
    \end{figure}
    It is not too hard to see that the triangulation embeds in the $K_4$-tiling, and that, as simplicial complexes, the $K_4$-tiling deformation retracts to the triangulation. 
    So the two simplicial complexes have isomorphic homology groups.  
    For the following calculation, therefore, which says that the basis of $H_1(\bold{P_{\ell} \times C})$ is a single gadget-cycle, we may assume that these two products coincide.  

   \begin{lemma}\label{lem:hcylinder}
     For any integer $\ell \geq 1$ and $d \geq 4$, and any $i \in [\ell]$,
     \[ H_1(\bbP_\ell \times \DsC) = \angles{i \times C} \cong \ZZ. \]  
   \end{lemma}
   \begin{proof}
     The  K\"unneth formula (see Theorem 3B.7 of \cite{Hatcher}) 
     gives us that 
     \[H_k(\bbP_\ell \times \DsC) = \bigoplus_{i+j=k} H_i(\bbP_\ell) \otimes_{\ZZ} H_j(\DsC) \]
     where $\oplus$ is the direct product of groups and $\otimes_{\ZZ}$ is the tensor product over $\ZZ$. 
   Thus we get 
     \begin{eqnarray*}
       H_1(\bbP_\ell \times \DsC) & = & (H_1(\bbP_\ell) \otimes_{\ZZ} H_0(\DsC))
                                      \oplus   (H_0(\bbP_\ell) \otimes_{\ZZ} H_1(\DsC)) \\
                           & = &  (0 \otimes_\ZZ \ZZ) \oplus (\ZZ \otimes_\ZZ \ZZ) \\
                           & = &  0 \oplus \ZZ = \ZZ,
     \end{eqnarray*}
     as needed. That the generator can be taken as any slice $i \times C$ seems clear. 
   \end{proof}

     Our main tools are excision for dealing with quotients, and Mayer-Vietoris sequences for patching together gadgets.
     For these we will use the {\em reduced homology groups} $\tilde{H}_i(\DG)$, which are defined on page 110 of \cite{Hatcher}. When $i \geq 1$ we have  $\tilde{H}_i(\DG) = H_i(\DG)$, but for $i = 0$ we have  $\tilde{H}_0(\DG) \oplus \ZZ \cong H_0(\DG)$. So whereas the dimension of $H_0(\DG)$ is the number of components of $G$, the dimension of $\tilde{H}_0(\DG)$ is one less than the number of components. Generally we use connected $G$ so $\tilde{H}_0(\DG)$ is usually $0$.
    
     Recall that a sequence of group homomorphisms
     \[
          \dots \to M_{i+2} \goes{\alpha_{i+2}}    M_{i+1} \goes{\alpha_{i+1}}  M_i \goes{\alpha_i} M_{i-1} \to \dots
     \]   
  is exact if $\alpha_{i} \circ \alpha_{i+1} = 0$  for each $i$.  If a we have a short exact sequence of groups
      \[ 0 \to A \to B \to C \to 0, \]
  where $C$ is free, then the sequence splits (see \cite{Hatcher} p. 148), which means that $B = A \oplus C$.

    The following is a special case of the process known as excision. It tells us that when contracting an edge not in a induced $C_4$, our basis is unchanged. 
   
   \begin{lemma}\label{lem:hexcision}
      Let $uv$ be an edge of $G$ that is in no induced copy of $C_4$, and let 
     $G' = G / uv$ be the graph  we get by contracting the edge $uv$. 
      Then $ [C]\mapsto [C / uv]$ is an isomorphism $H_1(\DG) \cong H_1(\DG')$.    
   \end{lemma}
   \begin{proof}
     As $uv$ is in no $C_4$, contracting it makes no new $K_3$, so no new $2$-simplices.
     Thus, as simplicial complexes, $\DG'$ is exactly the quotient $\DG/\DE$
     where $\DE$ is the simplicial complex of the edge $uv$. 
          
      Take an epsilon neighbourhood in $\DG$ around $\DE$. 
      It consists of $\DE$ and little nubs of edges and some frills of triangles. These
      all deformation retract to $\DE$ so we can apply excision
      (Theorem 2.13 of \cite{Hatcher}) to 
      get that 
    \[  H_1(\DE) \to H_1(\DG) \goes{j_*} H_1(\DG/\DE)  \to  \tilde{H}_0(\DE)   \]
    is exact, where $j_*$ is induced by the quotient $j: \DG \to \DG'$. 
    As $\tilde{H}_0(\DE) = 0$, and $H_1(\DE) = 0$, $j_*$ is an isomorphism $H_1(\DG) \cong H_1(\DG')$ as needed. 
   \end{proof}

    Finally, we recall the basic idea of Mayer-Vietoris sequences, and use it for several lemmas. 
    Given spaces $\DA$ and $\DB$ that intersect in $\DC = \DA \cap \DB$, we can often compute the homology of $\DX = \DA \cup \DB$ using the fact that the following sequence is exact. 
    
    \begin{align}\label{eq:MV}
        \dots \goes{\delta_{n+1}}& H_n(\DC) \goes{\Phi_n} H_n(\DA) \oplus H_n(\DB) \goes{\Psi_n}  H_n(\DX) \goes{\delta_n} 
            H_{n-1}(\DC) \goes{\Phi_{n-1}}  \\ \notag
        \dots \goes{\Psi_1} H_1(\DX) \goes{\delta_1}& 
            \tilde{H}_0(\DC) \goes{\Phi_0} \tilde{H}_0(\DA) \oplus \tilde{H}_n(\DB) \goes{\Psi_0} \tilde{H}_0(\DX) \goes{} 0 
    \end{align}
    where $\Psi(\alpha, \beta) = \alpha - \beta$.  See \cite{Hatcher} p. 149; and p. 150 where it states that \eqref{eq:MV} holds, as given above, for reduced homology.   
    
    This lemma, which we use a couple of times,  is nearly obvious.  But being the most simple use of a Mayer-Vietoris sequence, is a useful example.  
    
    \begin{lemma}\label{lem:hMV0}
      Let $G$ be a connected graph and $a$ and $b$ be vertices of $G$ at distance at least $4$ apart. 
      Where $G+E$ is the graph we get from $G$ by adding the edge $E = ab$, $H_1(\DG + \DE) = H_1(\DG) \oplus \ZZ$.
      Moreover one gets a basis of $H_1(\DG + \DE)$ from one of $H_1(\DG)$ by adding any cycle $C$ containing 
      one copy of the edge $E$. 
    \end{lemma}
    \begin{proof}
      As $a$ and $b$ are distance at least $4$ apart, we make no new triangles by adding $E$.
      So by \eqref{eq:MV} the following is exact. 
     \[ \to  H_1(\{\bba, \bbb\}) \to H_1(\DG) \oplus H_1(\DE) \goes{\Psi} H_1(\DG + \DE) \to \tilde{H}_0(\{\bba, \bbb\}) \to \tilde{H}_0(\DG) \oplus \tilde{H}_0(\DE) \to  \]
     Using $H_1(\{\bba, \bbb\}) = 0$, $H_1(\DE) = 0$ and $\tilde{H}_0(\DG) \oplus \tilde{H}_0(\DE) = 0 \oplus 0 = 0$ we get
     a short exact sequence
     \[ 0 \to  H_1(\DG) \oplus 0 \goes{\Psi} H_1(\DG + \DE) \goes{\delta} \tilde{H}_0(\{\bba, \bbb\}) \to 0. \]
      As $\tilde{H}_0(\{\bba, \bbb\}) = \ZZ$ is free, this splits to give
      $H_1(\DG + \DE) = H_1(\DG)  \oplus \tilde{H}_0(\{\bba, \bbb\}) = H_1(\DG) \oplus \ZZ$
      as needed.
      
      The statement about the basis comes as follows:
      the injection $\Psi$ injects $H_1(\DG)$ as a subgroup
      of $H_1(\DG + \DE)$. By exactness, this is exactly the kernel of $\delta$, so
      one gets a final basis element of $H_1(\DG + \DE)$ from any cycle that $\delta$ maps onto 
      the generator $[b-a]$ of $\tilde{H}_0({a,b})$.  This is any cycle made up of a walk in $G$ with boundary $[b-a]$ and a walk in $E$ with boundary $[a-b]$; so any cycle $C$ in $G$ that contains $E$ once.
    \end{proof}

    The following lemma uses the above, with excision to say that we `pinch' a long basis cycle into $2$ cycles, they both become basis cycles, replacing the original cycle in the basis.     
    This will be used in Step (2) of the construction of the sum gadget. 
     
   \begin{lemma}\label{lem:hMV1} 
     Let $G$ be a graph, $C$ be a nontrivial cycle in $H_1(\DG)$, and $a$ and $b$
     be vertices of $C$ at distance at least $4$ apart. Where $G'$ is the graph we get
     from $G$ by identifying $a$ and $b$, $H_1(\DG') = H_1(\DG) \oplus \ZZ$.
     Moreover, one gets a basis of $H_1(\DG')$ by taking a basis of $H_1(\DG)$, and doing either of:  
     \begin{itemize}
         \item adding one of the new cycles made from $C$ by identifying $a$ and $b$, or 
         \item replacing $[C]$ with the new cycles we get from $C$ by the identification.
     \end{itemize}
   \end{lemma}
   \begin{proof}
     Add an edge $E$ between $a$ and $b$. By Lemma \ref{lem:hexcision} and Lemma \ref{lem:hMV0} we get
         \[ H_1(\DG') = H_1(\DG + \DE) = H_1(\DG) \oplus \ZZ, \]
     as needed.  
     
     To get a basis of $H_1(\DG + \DE)$ from that of $H_1(\DG)$ we add any cycle containing $E$ once. So we us the cycle $C_1 = P_1 + E$ where $P_1$ is a path in $C$ from $b$ to $a$.  Moreover, where $C = P_2 + P_1$, we have $[C] = [C_2] + [C_1]$, so we get another basis of $H_1(\DG + \DE)$ by removing $[C]$ and adding $[C_2]$.  
     
     Under the the contraction of $E$ to $\DG'$, $C_1$ and $C_2$ become the cycles we get from $C$ by identifying $a$ and $b$, as needed.   
    \end{proof}

    The following says that if we identify end gadget-cycles of disjoint gadgets then,
    assuming that the identified gadget-cycle is in a basis for both gadgets, 
    our basis in the new graph is the union of the bases of the two gadgets.  (The identified gadget-cycle only appears once in the new basis.)
   
   \begin{lemma}\label{lem:hjoin1}   
     Let $F$ and $H$ be disjoint connected graphs and let $C_F$ and $C_H$ be
     isomorphic nontrivial $1$-cycles of $F$ and $H$ respectively.
     Let $G$ be the graph we get from $F$ and $H$ by identifying $C_F$ and $C_H$.
     Then
     \[ H_1(\DG)  = (H_1(\DF) \oplus H_1(\DH))/H_1(\DC_\DF).\]
     (This is '$=$', not just '$\cong$', so one gets a basis of $H_1(\DG)$ from
     the union of bases of $H_1(\DF)$ and $H_1(\DH)$, by removing the repeated basis element $[C_F]$. )       
   \end{lemma}
   \begin{proof}
     By \eqref{eq:MV}, the sequence
     \[ \dots \to H_2(\DG) \goes{\delta} H_1(\DC_\DF) \goes{\Phi} H_1(\DF) \oplus H_1(\DH) \goes{\Psi}
        H_1(\DG) \to \tilde{H}_0(\DC_\DF) \to \dots \]
     where $\Phi(\alpha) = (\alpha, -\alpha)$, 
     is exact.
      As $\tilde{H}_0(\DC_\DF) = \emptyset$, the map $\Psi$, which is isomorphic on
      both $H_1(\DF)$ and $H_1(\DH)$, is surjective. So elements of
      $H_1(\DF)$ and $H_1(\DH)$ generate $H_1(\DG)$.
      
      As $C_F$ is non-zero in $H_1(\DG)$ it is not the boundary of anything in $G$,
      so $\delta$  maps everything to $0$. Thus by exactness, $\Phi$ is injective,
      and the kernel of $\Psi$ is the element $(C_F,-C_F) = (C_F, - C_H)$, identifying
      $C_F$ and $C_H$.    
      So
      $H_1(\DG) = (H_1(\DF) \oplus H_1(\DH))/H_1(\DC_\DF)$,
      as needed.

    \end{proof}

    The following lemma deals with the situation that we identify two gadget-cycles of a connected graph.  In doing so, one of the identified gadget-cycles, both of which can be assumed to be basis cycles, becomes dependent on the other, so falls out of the basis unless they were already dependent; at the same time, we create a new basis cycle that we can take to be in $\sT$. 

   \begin{lemma}\label{lem:hMV2}
     Let $A$ and $B$ be disjoint nontrivial copies of $Z$ in a graph $G$, and let $\ccB$ be a basis of $H_1(\DG)$ containing $[A]$ and $[B]$.
     Let $G'$ be the graph that we get from $G'$ by identifying $A$ and $B$.   
     Let $C$ be any path of $G$  that becomes a cycle upon this identification.   
     \begin{enumerate}
          \item If $[A]$ and $[B]$ are independent in $H_1(\DG)$ then  
     \[ H_1(\DG') = \angles{ (\ccB \setminus \{B\}) \cup \{C\} }. \]
          \item If $[A] = -[B]$ (as might occur if we identify the cycles backwards) then 
      \[ H_1(\DG') = \angles{ (\ccB) \cup \{C\} }. \]
          \item If $[A] = [B]$ then 
            \[ H_1(\DG') = \angles{ (\ccB) \cup \{C\} }, \]
          and the element $[A]$ becomes $2$-torsion if it was not already.
      \end{enumerate}
      \end{lemma}
   \begin{proof}
     Let $E$ be a copy of $P_\ell \times A$ for $\ell \geq 4$, and let
     $G + E$ be the graph we get by identifying $\{0\} \times A$ in $E$ with $A$ and
     $\{\ell\} \times A$ in $E$ with $B$.  
     Using a similar trick to the one we used in Lemma \ref{lem:hMV1} we prove the result for $G+E$ instead of for $G'$.  Since we can get from $G+E$ to $G'$ by contracting lengthwise edges of $E$ one at a time, this is enough, by Lemma \ref{lem:hexcision}.

     As the intersection of $G$ and $E$ is $A \cup B$, the sequence
     \begin{eqnarray*}
     \goes{} H_1(\DA \cup \DB) & \goes{\Phi_1} & H_1(\DG) \oplus H_1(\DE)
                                             \goes{\Psi_1}  H_1(\DG+\DE) \\
        &\goes{\delta_1}& \tilde{H}_0(\DA \cup \DB) \goes{\Phi_0} \tilde{H}_0(\DG)\oplus \tilde{H}_0(\DE) \to \dots 
     \end{eqnarray*}
      from \eqref{eq:MV}, which quickly becomes
     \[ \to  H_1(\DA \cup \DB)  \goes{\Phi}  H_1(\DG) \oplus H_1(\DE) \goes{\Psi}  H_1(\DG+\DE) 
       \goes{\delta} \ZZ \to 0, \] 
     is exact. 
     
     Here $\Phi: \alpha \mapsto (\alpha, -\alpha)$ and
     $\Psi: (\alpha, \beta) \mapsto \alpha + \beta$. So we see that, $\Psi(H_1(\DG) \oplus H_1(\DE)) = \Psi(H_1(\DG) \oplus \{0\})$.
      Indeed, $H_1(\DE)$ has only the one non-zero element $[A] = [B]$
      and so for any $[C] \in H_1(\DG)$,
      \[ \Psi( ([C],[A])) = [C] + [A] = \Psi([C] + [A], 0). \]     
      
     With this we can show that the image of $\Psi$ is the image of its restriction $\Psi_r$ to $H_1(\DG) \oplus \{0\}$. Indeed, the kernel of $\Psi_r$
      is, by exactness 
      \begin{eqnarray*}
       \ker(\Psi) \cap ( H_1(\DG) \oplus \{0\}) & = & \img(\Phi) \cap (  H_1(\DG) \oplus \{0\}) \\
        & = &  \{ (\sigma, 0) \mid [\sigma] = 0 \mbox{ in } H_1(\DE) \}\\
        & = & \angles{( [A - B], 0 )}.  
      \end{eqnarray*}

      So $\Psi_r$, and so $\Psi$, injects $H_1(\DG) / \angles{[A - B]}$ into $H_1(\DG+\DE)$.  
      On the other hand, by exactness,  $\delta$ surjects onto $\ZZ$ so has co-rank of $1$, so $H_1(\DG+\DE)/\img(\Psi_r) \iso \ZZ$.  As $\ZZ$ is free, we get
      \[ H_1(\DG + \DE) \iso \img(\Psi_r) \oplus \ZZ = \frac{H_1(\DG)}{\angles{[A-B]}} \oplus \ZZ. \]

    The factor $\ZZ$ is generated by $[C]$ as $C$ is clearly  independent of anything in $H_1(\DG)$.   If $[A]$ and $[B]$ are independent in $H_1(\DG)$ then $A-B$ is non-zero and the factor $H_1(\DG)/\angles{[A-B]}$ is generated by the basis $\ccB \setminus \{[B]\}$; this gives part (1) of the lemma. 
     If $[A] = [B]$ then this factor is $H_1(\DG)$ so its basis is still $\ccB$.  If $[A] = -[B]$ it is then $H_1(\DG)/\angles{2[A]}$ and the basis is unchanged, but the order of $[A]$ is reduced to $2$.

    \end{proof}

    This final lemma says that if we `plug' a basis cycle $C$ in a graph with a vertex $v_C$ adjacent to all of its vertices, as we do in step (5) of Construction \ref{const:G*}, then $C$ becomes contractible and falls out of the basis. 
    
    \begin{lemma}\label{lem:plug}
     Let $C$ be a nontrivial cycle in a basis $\ccB$ of $H_1(\DG)$, and let $G'$ be the graph we get from $G$ by adding a new vertex $v_C$ adjacent to all vertices in $C$.
     The set  $\ccB \setminus \{ [C] \}$ is a basis of $H_1(\DG')$.
    \end{lemma}
    \begin{proof}
    
      Let $W$ be the wheel in $G'$ induced by $V(C) \cup \{v_C\}$.  It is easy to see that $H_1(\bbW) = 0$, and so by \eqref{eq:MV}, the sequence
       \[ \goes{\delta} H_1(\DC) \goes{\Phi} H_1(\DG) \oplus 0 \goes{\Psi} H_1(\DG') \to 0 \]
       is exact. So $\Psi$ surjects $H_1(\DC)$ onto $H_1(\DG')$. As $H_1(\DC)$ is clearly $\angles{[C]}$, and so the image of $\Phi$ is $\angles{ ([C],0)}$, this is the kernel of $\Psi$.  So 
        \[ H_1(\DG') \iso H_1(\DG)/\angles{[C]}. \]
    \end{proof}

   \subsection{Computing bases of $H_1(\bold{S_s})$ and $H_1(\bold{G_*})$}\label{sect:FinalTop}
   
    Throughout this section, it is convenient to refer to a basis of $H_1(\bold{G})$ for a graph $G$ simply as a basis of $G$.
    We prove Proposition \ref{prop:topo}, which says that $\Gb$ has a basis consisting of  end gadget-cycles.  We start by proving the same for the sum gadget $S_s$ of 
    Construction \ref{const:sum}. 

   \begin{lemma}\label{lem:SiBasis}
     For any $s \geq 2$, $H_1(\bold{S_s})$ has a basis consisting of any $s$ of the $s+1$ end gadget-cycles:  $\A_0, \A_1, \dots \A_{s-1}$ and $\ZS$.   
   \end{lemma}  
   \begin{proof}
     Referring to Construction \ref{const:sum}, in which we constructed $S_s$, we started with 
     $S_s^{(0)} = P_\ell \times \sB$, which by Lemma \ref{lem:hcylinder}  has a basis consisting of either of its end gadget-cycles. 
     
     In step (1) we take a quotient of one of these end gadget-cycles by subpaths.  Applying Lemma \ref{lem:hexcision} one edge at a time, we have that the resulting graph $S_s^{(1)}$ still has a basis consisting of either of its end gadget-cycles.   
     
     In step (2) we construct $S_s^{(2)}$ from $S_s^{(1)}$ by pinching several vertices together in the other end gadget-cycle $0 \times \sB$.  As $0 \times \sB$ can be assumed to be a basis cycle of $S_s^{(1)}$,  we can apply Lemma \ref{lem:hMV1} each time we do this, and conclude that the resulting cycles $\A_0', \dots, \A_{s-1}'$ make up a basis of $S_s^{(2)}$. 
     
     For each attachment, in step (3), of a path-of-cycles $P_1 \times Z$ to $\A_i'$, 
     we apply Lemma \ref{lem:hjoin1}, with $F = P_1 \times Z$ and with $H$ containing $C_H = \A_i'$; this allows us to replace $\A'_i$ with $\A_i$ in a basis. Thus the $\A_i$ make up a basis of $S_s = S_s^{(3)}$.   
     
     Lemma \ref{gad:sum} tells us that the end gadget-cycle $\ZS$ is independent of any $s-1$ of these $\A_i$, so can replace any one of them in the basis. 
   \end{proof}

   We are now ready to attack the main construction and show that $\Gb$ has a basis consisting of end gadget-cycles.   
   
   \begin{proof}[Proof of Proposition \ref{prop:topo}]
   
     In Step (1) of Construction \ref{const:G*} we introduce a set of $3n+1$ disjoint copies
     of $Z$. Indeed, letting $V(G) = [n]$ and, for $i \in [n]$, letting $\ccB^i = \{ \A^i_0, \A^i_1, \A^i_2 \}$, we introduce the set 
       \[ \ccB := \bigcup_{i \in [n]} \ccB^i  \cup \{ \sZ \} \]
     of copies of $Z$.    
     
     In Step (2) we construct a graph $\Ga[2]$, the end gadget-cycles of which are the cycles in the set $\ccB$, as follows. We start with a copy $G_1:= S^1_3$ of $S_3$ and rename its copy of $\ZS$ to $\sZ$.   For $i = 2, \dots, n$, we construct $G_i$ from $G_{i-1}$ by taking a new copy $S^i_3$ of $S_3$ and identifying its copy of $\ZS$ with $\sZ$.  The resulting $G_n$ is $\Ga[2]$.  
     
     Now, by Lemma \ref{lem:SiBasis}, we can assume that the copy of $\ZS$ in each of the $S^i_3$ is in its basis, so applying Lemma \ref{lem:hjoin1} when we construct $G_{i+1}$ from $G_i$ we get a basis of
     $G_{i+1}$ by adding any two of the cycles in $\ccB^{i+1}$ to our basis of $G_i$.  It follows that for any choice of two cycles in $\ccB^i$ for each $i$, there is a basis of $\Ga[2]$ containing these cycles.  
     At this point $\sT$ is still a tree, so has an empty basis $\ccB_T$.  The following statement is therefore true:
     \begin{quote}
         For any two cycles $C, C'$ in $\ccB$ there is a basis $\ccB^*$ of $\Ga[2]$ with 
          \begin{equation}\label{eq:basis} \{C, C'\} \cup \ccB_T \subset \ccB^* \subset \ccB \cup \ccB_T. \end{equation}
     \end{quote}

     In Steps (3) and (4) we construct $\Ga[4]$ from $\Ga[2]$ by taking several new copies of $S_2$ and attaching them to $\Ga[2]$ by identifying their copies of $\A_0$ and $\A_1$ with cycles of $\Ga[2]$ in $\ccB$. Applying the following claim each time we do this, we add a new cycle of $\sT$ to $\ccB_T$ and verify that \eqref{eq:basis} holds for the newly constructed graph in place of $\Ga[2]$.  
     
     \begin{claim*}
      Let $G$ be a graph and $C$ and $C'$ be disjoint copies of $Z$ in a basis $\ccB$ of $G$.  Let $T$ be a connected subgraph of $G$ with one vertex in each of $C$ and $C'$.  Let $G'$ be the graph we get from $G$ and a disjoint copy $S$ of $S_2$ by identifying $C$ with $\A_0 \leq S$ and $C'$ with $\A_1 \leq S$ in such a way that the tree $T_S$ of $S$ has one leaf identified with the vertex of $T$ in $C$ and one other leaf the vertex of $T$ in $C'$.  Let $T'$ be the union of $T$ and $T_S$.  
      
      Where $\ccB_T$ is a basis of $T$, there is a cycle $C_T$ in $T'$ such that 
       \begin{enumerate}
           \item $\ccB_T \cup \{C_T\}$ is a basis of $T'$, and 
           \item $\ccB \cup \{C_T\}$ is a basis of $G'$.
       \end{enumerate}
     \end{claim*}
     \begin{proof}\claimproof
     Starting with $G$ and the copy $S$ of $S_2$, first identify $\A_0$ in $S$ with the cycle $C$ of $G$.
     Call this graph $G^{(1)}$. 
     By Lemma \ref{lem:hjoin1}, we get a basis $\ccB^{(1)}$ of $G^{(1)}$ by removing $C$ from $\ccB$ and adding the cycles $\A_1$ and $\ZS$ of $S$.  The basis of $T$ has not changed, as we have just attached a tree $T_S$ to some vertex.  
     
     Now construct $G'$ by identifying the cycles $\A_1$ and $C'$ of $G^{(1)}$.  This makes $T'$ by identifying vertices of
     $T$ that are distance at least $4$ apart, so by Lemma \ref{lem:hMV1} creates a new cycle $C_T$ in $T'$ such that 
     $\ccB_T \cup \{C_T\}$ is a basis of $T'$.  
     By Lemma \ref{lem:hMV2}, we can get a basis of $G'$ from $\ccB^{(1)}$ by adding the cycle $C_T$ and possibly  removing the cycle $C'$.  This basis contains $\ZS$ and $\A_1$ of $S$. But in $S$, the cycles $\ZS, \A_1,$ and $\A_0$ are minimally dependent,
     so we can replace $\ZS$ with $\A_0 = C$ in this basis, getting that $\ccB \cup \{C_T\}$ is a basis of $G'$.
     \end{proof} 
     
     As mentioned before the claim, this gives us that statement \eqref{eq:basis}  holds for $\Ga[4]$ in place of $\Ga[2]$.  In particular $\Ga[4]$ has a basis consisting of gadget end-cycles, and cycles in $\ccB_T$. To get $\Gb = \Ga[5]$ from $\Ga[4]$ we added, for every cycle $C$ in $\ccB_T$, a vertex $v_C$ with edges to all vertices of $C$.  By Lemma \ref{lem:plug} these cycles $C$ drop out of the basis, and so $\Gb$ has a basis consisting only of gadget end-cycles.    
   \end{proof}


\end{document}